\documentclass[11pt,reqno,twoside]{amsart}
\usepackage{graphicx}
\usepackage{amssymb}
\usepackage{epstopdf}
\usepackage[asymmetric,top=3.5cm,bottom=4.3cm,left=3.1cm,right=3.1cm]{geometry}
\usepackage{graphicx}
\usepackage{bm}

\usepackage{booktabs} 
\usepackage{array} 
\usepackage{paralist} 
\usepackage{verbatim} 
\usepackage{subfigure} 
\usepackage{tabularx}
\usepackage{amsmath,amsfonts,amsthm,mathrsfs,amssymb,cite}
\usepackage[usenames]{color}
\usepackage{ulem}

\newtheorem{thm}{Theorem}[section]

\newtheorem{lem}{Lemma}[section]
\newtheorem{prop}{Proposition}[section]
\theoremstyle{definition}
\newtheorem{defn}{Definition}[section]
\theoremstyle{remark}

\newtheorem{rem}{Remark}[section]
\numberwithin{equation}{section}

\newcommand{\bu}{\mathbf{u}}

\newcommand{\bv}{\mathbf{v}}
\newcommand{\bw}{\mathbf{w}}
\newcommand{\bff}{\mathbf{f}}

\newcommand{\bx}{\mathbf{x}}

\newcommand{\bh}{\mathbf{h}}
\newcommand{\bg}{\mathbf{g}}

\newcommand{\bphi}{\bm{\phi}}

\newcommand{\bnu}{\bm{\nu}}

\newcommand{\eqnref}[1]{(\ref {#1})}

\def\Oh{{\mathcal O}}

\newcommand{\beq}{\begin{equation}}
\newcommand{\eeq}{\end{equation}}
\allowdisplaybreaks

\title[]
{Dislocations with corners in an elastic body with applications to fault detection}

\author{Huaian Diao}
\address{School of Mathematics and Key Laboratory of Symbolic Computation and Knowledge Engineering of Ministry of Education, Jilin University, Changchun, Jilin, China.}
\email{diao@jlu.edu.cn, hadiao@gmail.com}
\author{Hongyu Liu}
\address{Department of Mathematics, City University of Hong Kong, Kowloon Tong, Hong Kong SAR, China.}
\email{hongyu.liuip@gmail.com; hongyliu@cityu.edu.hk}
\author{Qingle Meng}
\address{Department of Mathematics, City University of Hong Kong, Kowloon Tong, Hong Kong SAR, China.}
\email{mengql2021@foxmail.com; qinmeng@cityu.edu.hk}
\begin{document}
\maketitle

\begin{abstract}
This paper focuses on an elastic dislocation problem that is motivated by applications in the geophysical and seismological communities. In our model, the displacement satisfies the Lam\'e system in a bounded domain with a mixed homogeneous boundary condition. We also allow the occurrence of discontinuities in both the displacement and traction fields on the fault curve/surface. By the variational approach, we first prove the well-posedness of the direct dislocation problem in a rather general setting with the Lam\'e parameters being real-valued $L^\infty$ functions and satisfy the strong convexity condition. Next, by considering that the Lam\'e parameters are constant and the fault curve/surface possesses certain corner singularities, we establish a local characterization of the jump vectors at the corner points over the dislocation curve/surface. In our study, the dislocation is geometrically rather general and may be open or closed. We establish the unique results for the inverse problem of determining the dislocation curve/surface and the jump vectors for both cases.
\medskip

\noindent{\bf Keywords:}~~dislocations, elasticity, corners, slips, well-posedness, inverse problem, uniqueness

\medskip

\noindent{\bf MSC codes:}~~74B05, 94C12, 35A02

\end{abstract}

\section{Introduction}

In this work, we investigate an elastic dislocation problem motivated by applications of crystalline materials, geophysics, and seismology. Building upon the previous studies reported in \cite{BonafedeRivalta1999, Cohen1996, Eshelby1973, Jiang2014, Rivalta2002, Segall2010, XuXuYu2013} and the related literature cited therein, we adopt an elastostatic system with mixed boundary conditions (Problem \eqref{eq:elast1}) to describe our model within an isotropic, homogeneous, and bounded domain $\Omega \subset \mathbb{R}^n$ ($n=2,3$). A distinguishing feature of this model is the allowance for discontinuities in both the displacement field and traction field along the fault curve/surface denoted by $\mathcal{S} \subset \Omega$. We employ the notations $\bff$  and $\bg$ to represent these discontinuities, which are the jump functions associated with the fault throughout this paper. Especially, $\bff$ is also known as a slip in the literature (cf.\cite{AspriBerettaMazzucatoHoop2020, Beretta2008}). Our investigation includes both the direct and inverse dislocation problems. In the direct problem, we aim to determine the elastic displacements of the elastic transmission problem with mixed homogeneous boundary conditions (see Problem \eqref{eq:elast1}), given that the dislocation curve/surface $\mathcal{S}$, the elastic stiffness tensor $\mathbb{C}(\bx)$, and the jump functions $\bff$ and $\bg$ over $\mathcal{S}$ are known a-priori. Conversely, the inverse problem seeks to determine the dislocation curve/surface $\mathcal{S}$ and the jump functions $\bff$ and $\bg$ from displacement field measurements with some regularities requirements about the dislocation curve/surface, and these jump functions. For the inverse dislocation problems, we focus on a single measurement of displacement, which implies that we only require a pair of the Cauchy data on the partial boundary of $\partial \Omega$ to determine $\mathcal{S}$, $\bff$, and $\bg$ with some priori information.

 It is widely assumed in the field of seismic propagation that the displacement has discontinuities, while the traction field remains continuous. The simplest assumption enables a linear relationship between slip and stress traction in the model, as exemplified in \cite{CS1995}. However, other studies, such as the literature \cite{M1949, BM1985, V2009, Pichierri2020, Ar2002book}, specifically address scenarios where the traction field also has discontinuities across the dislocation. Indeed, in the 2D linear elastic model \cite{Pichierri2020}, the displacement and traction fields experience jumps, where the fault is a straight interface. In this research, we aim to explore the theoretical results of a relatively complex model that allows discontinuities in both the displacement and traction fields along the dislocation in $\mathbb{R}^n\,(n=2,3)$. Additionally, we consider dislocations as either open or closed curves or surfaces. Open dislocations have many practical applications, such as creeping faults (cf. \cite{AspriBerettaMazzucato2022}).
In contrast, closed faults in an anisotropic inhomogeneous elastic media are studied in \cite{PowW2021}, where the corresponding well-posedness are derived using the integral equation method. Furthermore, characterizations of non-radiating faults are investigated in \cite{PowW2021}.  This paper also explores closed faults from a mathematical perspective using variational techniques. However, practical applications for closed dislocations have yet to be identified. Our theoretical research may serve as a foundation for future real-world implementations.

There are extensive studies on elastic dislocations in the half-space $\mathbb{R}^n_+$, which can be found  \cite{Volkov2017, VS2019, TV2020, AspriBerettaMazzucatoHoop2020} and the related references cited therein. Indeed,  there are extensive studies on both the direct and inverse problems of elastic dislocations in developing numerical reconstruction algorithms for determining the surface of the dislocations and the associated displacement field from boundary measurements in the half-space  $\mathbb{R}^n_+$ \cite{Volkov2017, VS2019, TV2020, AspriBerettaMazzucatoHoop2020}. For example,  \cite{Volkov2017} studied the linear plane elasticity model and proposed an iterative method to detect the fault plane and unidirectional tangential slip, where the iterative procedure is based on Newton's methods or constrained minimization of a suitable misfit functional by using a finite number of surface measurements. Additionally, \cite{VS2019} considered numerical reconstruction algorithms and Bayesian approaches to quantify uncertainties. Some related discussions on reconstruction and the corresponding non-quantitative stability estimates with constant Lam\'{e} coefficients can be referred to \cite{TV2020}. Recently,  \cite{AspriBerettaMazzucatoHoop2020} also investigated the elastic dislocation problems in the half-space under the assumptions of Lipschitz continuous elastic coefficients and a Lipschitz fault surface. Uniqueness result was derived  in determining the fault geometry and associated slip from the measurement of surface displacement on an open set, provided that the fault has at least one corner singularity, can be represented as a graph in a given coordinate system, and the slip is either tangential or normal to the fault. Other investigations have addressed the dislocation problem within bounded domains, such as those in \cite{Beretta2008, Zwieten2014}, expanding beyond the half-space setting. Recently,  \cite{AspriBerettaMazzucato2022} considered an anisotropic and inhomogeneous elastic medium in 2D, while \cite{AspriBerettaMazzucato20222} extended this to a 3D scenario with elastic dislocations represented as open, oriented fault surfaces in an elastostatic system, featuring discontinuities in the displacement field across these faults. The requirement for the slip to belong to a functional space with good extension properties is crucial for solving the direct problem using variational methods. The authors established the uniqueness of determining both the fault surfaces and the associated slip vectors, while relaxing the regularity requirements to merely Lipschitz continuity, with the fault surfaces represented as graphs in a fixed coordinate system.

 For the direct problem of classical elastic dislocations, it is common practice to employ a lift of the jumps to establish a variational framework that ensures its well-posedness in bounded and unbounded domains. This approach has been extensively discussed in the context of dislocations (cf. \cite{Zwieten2014}). In this paper, our focus is on dislocation problems with certain corner singularities within a bounded domain. There exists a substantial body of work on the singularity analysis of corner transmission problems (cf. \cite{Costabel1993, Dauge1988,DLMW20252, Nicaise1992}), as well as the application of weighted spaces in transmission problems (cf. \cite{Li2013}).  To ensure the consistency of this article, we provide proof of the well-posedness of Problem \eqref{eq:elast1}, where there are jumps in both the displacement and traction fields across the fault curve/surface. Employing a variational approach, we establish the well-posedness of the direct dislocation problem in general scenarios, which includes a particular setting where the Lam\'e  parameters are real-valued $L^\infty$ functions and satisfy the strong convexity condition. Furthermore, we consider both open and closed fault curves/surfaces.

In the context of inverse dislocation problems, it is noted that alternative approaches and additional numerical analyses have been proposed, particularly for polyhedral domains (cf. \cite{Volkov2017}). However, our study focuses on scenarios where the Lam\'e parameters are constant and the fault curve or surface exhibits specific corner singularities (namely, 2D corners/3D edge corners). When $\mathcal{S}$ is closed, it forms the boundary of a convex polygon or convex polyhedron. If $\mathcal{S}$ is open, it is composed of linear piecewise curves in $\mathbb{R}^2$ or piecewise surfaces in $\mathbb{R}^3$, as outlined in Definition \ref{def:poly}. This a-priori information on the geometry of $\mathcal{S}$ is essential for establishing the corresponding uniqueness result from a single measurement.  As pointed out in \cite{CK,CK2018}, achieving uniqueness in inverse scattering problems from a single measurement is highly challenging, and certain a-priori information on the geometry is unavoidable.  Related studies have explored inclusions with corners and semilinear terms using minimal measurements \cite{DFLW2025}. Furthermore, effective electromagnetic scattering problems have been analyzed in \cite{DLMW20251}. Unique shape identification for inverse elastic and electromagnetic scattering problems in layered polygonal or corona-type media, based on a single far-field measurement, has been investigated in \cite{DFLY24,DTLT24}.

To characterize the jump functions at the corner points along the dislocation curve or surface, we conduct a local singularity analysis of the elastic field around these points using a microlocal approach. This requires certain H\"older regularity assumptions on the slip and traction jumps around the underlying corner (detailed in Definition \ref{def:Admis2}). In the 3D case, an additional assumption that both the slip field and the traction jump are independent of one spatial variable is necessary, as clarified in Definition \ref{def:Admis2}. We employ Complex Geometric Optic (CGO) solutions for the underlying elastic system to characterize the jump vectors at 2D corner and 3D edge points. Our study involves intricate and delicate analysis, and we establish unique results for the inverse problem of determining the dislocation curve or surface and the associated jumps from a single measurement, whether the fault curves or surfaces are open or closed.  Our recent work \cite{DLM2025} establishes that dislocations, their associated jump functions, and interfaces in multi-layered elastic solids can also be uniquely determined from a single passive measurement. 

The paper is structured as follows. In Section~\ref{se:0}, we introduce the mathematical setup of our study and establish the well-posedness of the dislocation problem. In Section~\ref{sec:1}, we propose some admissible assumptions about $(\mathcal{S};\bff,\bg)$ for the scenario that $\mathcal{S}$ may be closed or open. Assuming that Lam\'e parameters are constants and the fault curve/surface possesses certain corner singularities, we establish a local characterization of the jump vectors at the certain corner points over the dislocation curve/surface. Furthermore, we also develop global uniqueness results for the inverse dislocation problem for determining the dislocation curve/surface $\mathcal{S}$ and the jump vectors $\bff$ and $\bg$ in these two cases with the additional geometrical assumption about the dislocation curve/surface. In Section~\ref{se:pre}, we derive several local results for the jump vectors $\bff$ and $\bg$ at these certain corner points along $\mathcal{S}$. Section~\ref{se:proof} is devoted to proving the uniqueness of the inverse problem presented in Section~\ref{sec:1}.

\section{Mathematical setup and the direct problem}\label{se:0}
In this section, we pay attention to the mathematical setup of the dislocation problem and the study of the well-posedness of the direct problem.

\subsection{Mathematical setup}
We first introduce a geometric and mathematical setup for our study; see Fig.~\ref{fig:1} for a schematic illustration in 2D.
\begin{figure}[ht]
\centering
    {\includegraphics[width=0.8\textwidth]{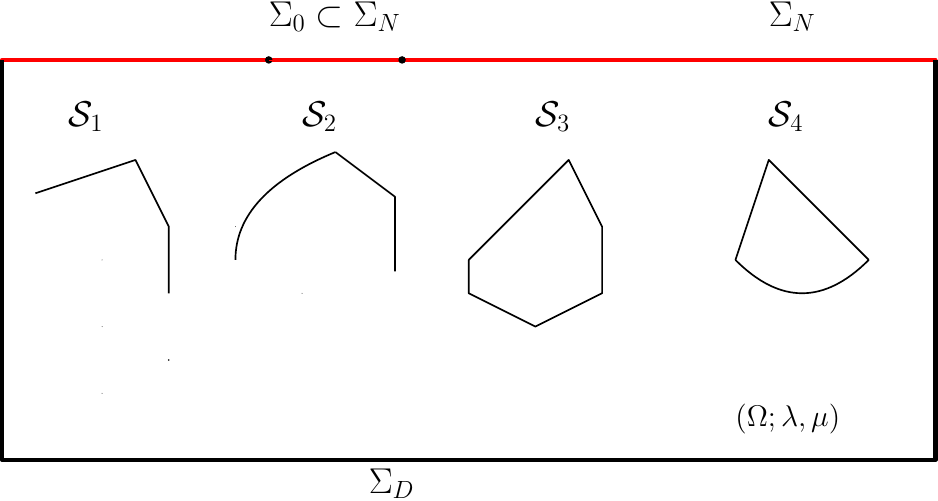}}
    \caption{Schematic illustration of  an elastic dislocation, where $\mathcal S_i\ (i=1,\ldots,4)$ are fault curves and $\partial \Omega=\Sigma_D \cup \Sigma_N$.\label{fig:1}  }
\end{figure}
Let $\lambda(\mathbf{x})$ and $\mu(\mathbf{x})$, $\mathbf{x}=(x_j)_{j=1}^n\in\Omega$, be real-valued $L^\infty$ functions, which are referred to as the Lam\'e parameters of the elastic solid $\Omega$. We define $\mathbb{C}(\bx)=(C_{ijkl}(\mathbf{x}))_{i,j,k,l=1}^n$, $\bx\in\Omega$, as a four-rank tensor given by:
\begin{equation}\label{eq:tensor1}
\mathbb{C}(\mathbf{x}):=\lambda(\mathbf{x})\mathbf{I}\otimes\mathbf{I}+2\mu(\mathbf{x})\mathbb{I},\ \ \text{where}\ \ C_{ijkl}(\mathbf{x})=\lambda(\mathbf{x})\delta_{ij}\delta_{kl}+\mu(\mathbf{x})(\delta_{ik}\delta_{jl}+\delta_{il}\delta_{jk}),
\end{equation} Here, $\mathbf{I}$ and $\mathbb{I}$ respectively represent the identity operators of the second- and fourth-rank tensors, and $\delta_{ij}$ is the Kronecker delta. We assume that the Lam\'e parameters $\lambda(\mathbf{x})$ and $\mu(\mathbf{x})$ satisfy the strong convexity conditions:
\begin{equation}\label{eq:tensor2}
\mu(\mathbf{x})> 0, \quad 2\mu(\mathbf{x})+n\lambda(\mathbf{x})>0,\quad \forall \mathbf{x}\in\Omega,
\end{equation}
These conditions ensure the uniform strong convexity of the elastic stiffness tensor $\mathbb{C}(\bx)$. For $\mathbf{A}=(a_{ij})_{i,j=1}^n$, we define the operation ``$:$" as:
\begin{equation}\label{eq:op1}
\mathbb{C}:\mathbf{A}=\big(C_{ijkl}:\mathbf{A}\big)_{ij}=(\Pi)_{ij},\quad\text{where}\quad \Pi_{ij}:=\sum_{k,l=1}^n C_{ijkl} a_{kl}.
\end{equation} Let $\mathbf{u}(\bx)=(u_j(\bx))_{j=1}^n$, $\bx\in\Omega$, with each $u_j(\bx)$ being a complex-valued function. We introduce the Lam\'e operator $\mathcal{L}$ as:
\begin{align*}
\mathcal{L}\mathbf{u}:=\nabla\cdot(\mathbb{C}:\nabla \mathbf{u})=\mu\Delta\bu+(\lambda+\mu)\nabla(\nabla\cdot\bu),
\end{align*} 
where $\nabla \mathbf{u}:=(\partial_j u_i)_{i,j=1}^{n}$, $\nabla\cdot \bu :=\sum^{n}_{j=1}\partial_j u_j$, and $\partial_ju_j=\partial u_j/\partial x_j$. Furthermore, let $\bnu\in\mathbb{S}^{n}$ be the unit normal vector, and we define the traction operator as:
\begin{align*}
\mathcal{T}_{\bnu}(\mathbf{u})=\bnu\cdot(\mathbb{C}:\nabla\mathbf{u}).
\end{align*}

Let $\mathcal{S}\subset\Omega$ be an oriented Lipschitz curve/surface, which can be open or closed. We define $\mathcal{S}^\pm$ as the two sides of $\mathcal{S}$, with $\mathcal{S}^+$ representing the side where $\bnu$ points outward. We denote the jump of a function or tensor field $\mathbf{p}$ across $\mathcal{S}$ as $[\mathbf{p}]_\mathcal{S}:=\mathbf{p}|_{\mathcal{S}}^+-\mathbf{p}|_{\mathcal{S}}^-$, where $\mathbf{p}|_{\mathcal{S}}^\pm$ represent the non-tangential limits of $\mathbf{p}$ on $\mathcal{S}^\pm$, respectively. The elastic dislocation problem that we consider allows the occurrence of discontinuities in both the displacement and traction fields, denoted by $\bff$ and $\bg$, respectively. 

\subsection{Mathematical model}

In this paper, our main focus is on the following elastostatic system for $\mathbf{u}\in H^1(\Omega\backslash\overline{\mathcal{S}})^n$:
\begin{alignat}{2}\label{eq:elast1}
\begin{cases} \mathcal{L}\mathbf{u}(\mathbf{x})=\mathbf{0},& \mathbf{x}\in\Omega\backslash\overline{\mathcal{S}},\\
\mathcal{T}_{\bnu}\mathbf{u}\big|_{\Sigma_N}=\mathbf{0},&\mathbf{u}\big|_{\Sigma_D}=\mathbf{0},\\
[\mathbf{u}]_\mathcal{S}=\mathbf{f},&[\mathcal{T}_{\bnu} \mathbf{u}]_\mathcal{S}=\mathbf{g}.
\end{cases}
\end{alignat} Here, $\partial\Omega=\Sigma_D\cup\Sigma_N$ represents a Lipschitz partition of $\partial\Omega$.
To investigate both the direct and inverse dislocation problems, we introduce some relevant function spaces for the jumps $\bff$ and $\bg$. Note that the function spaces differ depending on whether $\mathcal{S}$ is closed or open.

\medskip {\noindent {\bf Class 1:}
When $\mathcal{S}$ is closed, we consider jump functions $\bff$ and $\bg$ that satisfy
\begin{equation*} \mathbf{f}\in H^{\frac{1}{2}}(\mathcal{S})^n\quad \text{and} \quad \mathbf{g}\in H^{-\frac{1}{2}}(\mathcal{S})^n.
\end{equation*}

\medskip
\noindent {\bf Class 2:}
When $\mathcal{S}$ is open, we assume that $\mathbf{f}$ and $\mathbf{g}$ belong to appropriate weighted spaces with a good extension property on the curve/surface, which is standard in the literature (cf. \cite{AspriBerettaMazzucato2022, AspriBerettaMazzucato20222, Lions1973}). For the sake of the completeness of the paper, we provide the detailed proof in the following. Let $H_{0}^{\frac{1}{2}}(\mathcal{S})^n$ denote the closure of $ C_0^\infty(\mathcal{S})$ with respect to the norm $\|\cdot\|_{H^{\frac{1}{2}}(\mathcal{S})^n}$. For the case that $\mathcal{S}$ is an open bounded Lipschitz domain in $\mathbb{R}^n$, it is well known that $H_{0}^{\frac{1}{2}}(\mathcal{S})^n=H^{\frac{1}{2}}(\mathcal{S})^n$ and the extension operator from $H^{1/2}_0(\mathcal{S})$ to $H^{1/2}_0(\mathcal{\mathbb{R}}^n)$ is not continuous. Then
 we introduce the so-called {\it Lions-Magenes} space $H_{00}^{\frac{1}{2}}(\mathcal{S})^n$ (see \cite{Lions1973}) as follows:
\begin{align*}
H_{00}^{\frac{1}{2}}(\mathcal{S})^n&:=\Big\{ \bu\in H_{0}^{\frac{1}{2}}(\mathcal{S})^n;\,\,\varrho^{-\frac{1}{2}} \bu\in L^2(\mathcal{S})^n      \Big\}, 
\end{align*}
where $\varrho\in C^\infty(\overline{\mathcal{S}})$ denotes a weight function with these properties: (1) $\varrho$ has the same order as the distance to the boundary (i.e., $\lim\limits_{\bx\rightarrow \bx_0}\frac{\varrho(\bx)}{\mathrm{d}(\bx,\partial \mathcal{S})}=d\neq 0,\, \forall \,\bx_0\,\in \partial \mathcal{S}$); and (2) $\varrho(\bx)$ is positive in $\mathcal{S}$ and $\varrho$ vanishes on $\partial \mathcal{S}$.
This space is equipped with the norm
\begin{alignat*}{2}
\|\bu\|_{H_{00}^{\frac{1}{2}}(\mathcal{S})^n}&=\|\bu\|_{H^{\frac{1}{2}}(\mathcal{S})^n}+\|\varrho^{-\frac{1}{2}}\bu\|_{L^2(\mathcal{S})^n}
&\,\quad\mbox{for}\,\,&\bu\,\in H_{00}^{\frac{1}{2}}(\mathcal{S})^n.
\end{alignat*}}
In addition, $\bg$ belongs to the space $ H^{-\frac{1}{2}}_0(\mathcal{S})^n$  which is the dual space of $H_{00}^{\frac{1}{2}}(\mathcal{S})^n$.
More details about the weighted spaces can be found in \cite{Mclean2010, Tartar2006, Cessenat1998}.
Indeed, when $\mathcal{S}$ is a curve in the 2D case, $\partial \mathcal{S}$ corresponds to the set of endpoints of $\mathcal{S}$. Similarly, when $\mathcal{S}$ is a surface in the 3D case, $\partial \mathcal{S}$ corresponds to the boundary curve of $\mathcal{S}$. As discussed in \cite{AspriBerettaMazzucato2022, AspriBerettaMazzucato20222, Tartar2006}, we know that $H^{1/2}_{00}(\mathcal{S})$ is the optimal subspace of the space $H^{1/2}(\Gamma)$, whose element can be continuously extended by zero to a component of $H^{1/2}(\Gamma)$. Here, $\Gamma=\overline{\mathcal{S}}\cup\Gamma_0$ is a closed Lipschitz curve/surface extended by $\mathcal{S}$ and 
satisfying $\Gamma\cap \partial \Omega=\emptyset$, where $\Gamma_0$ is a curve or a surface linking with the boundary $\partial\mathcal{S}$ satisfying $ \Gamma_0\cap( \mathcal{S} \backslash \partial \mathcal{S}) =\emptyset $. Hence, the Lipschitz domain $\Omega$ can be partitioned into two connected subdomains $\Omega_1$ and $\Omega^c_1=\Omega\backslash\overline{\Omega}_1$, where $\partial \Omega_1=\Gamma$ and $\partial \Omega^c_1=\Gamma\cup \partial \Omega$.
Let $\bff$ be continuously extended to $\widetilde{\bff}\in H^{1/2}(\Gamma)^n$ by zero on $\Gamma \backslash \mathcal{S}$ and $\bg$ be continuously extended to $\widetilde{\bg}\in H^{-\frac{1}{2}}(\Gamma)^n$ by zero on $\Gamma \backslash \mathcal{S}$. That is to say
\begin{equation}\label{eq:f,g}
\widetilde{\bff}=\begin{cases}
\bff, \quad \bx\in \mathcal{S},\\
{\bf 0},\quad \bx \in \Gamma\backslash \overline{\mathcal{S}}
\end{cases}
\quad \mbox{and}\qquad
\widetilde{\bg}=\begin{cases}
\bg, \quad \bx\in \mathcal{S},\\
  {\bf 0},\quad \bx \in \Gamma\backslash \overline{\mathcal{S}}.
\end{cases}
\end{equation}
In particular, when $\mathcal{S}$ is closed, $\Gamma=\mathcal{S}$, $\widetilde{\bff}=\bff$ and $\widetilde{\bg}=\bg$.
 Regardless of whether $\mathcal{S}$ is open or closed, we use
\begin{align}\label{eq:omega1}
	\Omega_1:= {\mathrm {enclose}}(\mathcal{S})
\end{align}
 to denote the aforementioned domain satisfying $\partial \Omega_1=\Gamma$. Similarly, when $\mathcal{S}$ is closed, let the domain ${\mathrm {enclose}}(\mathcal{S})$ satisfy that $\partial \left(\mathrm {enclose}(\mathcal{S})\right) =\mathcal{S}. $

We now are in the position to show the existence of a unique weak solution $\mathbf{u}\in H^1_{\Sigma_D,\Sigma_N}(\Omega\backslash\overline{\mathcal{S}})^n$ to Problem \eqref{eq:elast1} corresponding with the boundary data on $\Sigma_D$ and $\Sigma_N$.
\begin{thm}
There exists a unique solution $\bu\in H_{\Sigma_D,\Sigma_N}^1(\Omega\backslash\overline{\mathcal{S}})^n$ to  Problem~\eqref{eq:elast1}.
\end{thm}

\begin{proof}
In the following, we first prove the well-posedness of Problem \eqref{eq:elast1} for the case that $\mathcal{S}$ is open. When $\mathcal{S}$ is closed, the corresponding proof can be obtained similarly.  We shall use the variational technique (cf.\cite{Hahenr2000}) to verify that there exists a unique solution to this problem. From \cite[Lemma~3.2]{AspriBerettaMazzucato2022} and \cite [Remark~3.3]{AspriBerettaMazzucato20222}, Problem ~\eqref{eq:elast1} can be recast equivalently as the following PDE system for $\bu_1\in H^1(\Omega_1)^n$ and $\bu_2\in H^1(\Omega^c_1)^n$ such that
\begin{align}\label{eq:elast1'}
\begin{cases}
\mathcal{L}\mathbf{u}_1(\mathbf{x})=\mathbf{0},\quad \mathbf{x}\in\Omega_1,\\
\mathcal{L}\mathbf{u}_2(\mathbf{x})=\mathbf{0},\quad \mathbf{x}\in\Omega^c_1,\\
\mathcal{T}_{\bnu}\mathbf{u}_2\big|_{\Sigma_N}=\mathbf{0},\,\,\mathbf{u}_2\big|_{\Sigma_D}=\mathbf{0},\\
\bu_2\big|_\Gamma-\bu_1\big|_\Gamma=\widetilde{\mathbf{f}},
\\ \mathcal{T}_{\bnu} \mathbf{u}_2\big|_\Gamma-\mathcal{T}_{\bnu} \mathbf{u}_1\big|_\Gamma=\widetilde{\mathbf{g}},
\end{cases}
\end{align}
where $\widetilde{\bff}\in H^{1/2}(\Gamma)^n$ and $\widetilde{\bg}\in H^{-1/2}(\Gamma)^n$ defined by \eqref{eq:f,g} are given. Let $\bu_{\,\widetilde{\bff}}$ be the unique solution to the following Dirichlet boundary value problem
\begin{equation*}
\mathcal{L}\,\mathbf{u}_{\,\widetilde{\bff}}=\mathbf{0}  \quad\mbox{in}\quad\Omega^c_1,\quad
\mathbf{u}_{\,\widetilde{\bff}}=\widetilde{\bff} \quad\mbox{on}\quad \Gamma,\quad
\mathbf{u}_{\,\widetilde{\bff}}={\bf 0}\quad\mbox{on}\quad {\partial \Omega}.
\end{equation*}
Let us next consider the variational formulation of Problem~\eqref{eq:elast1'}: Find $\bw\in H_{\Sigma_D,\Sigma_N}^1(\Omega)^n$ satisfying
\begin{align}
&\int_{\Omega_1}(\mathbb{C}:\nabla \mathbf{w}):\nabla\overline{\bphi}\,\mathrm d\mathrm \bx+\int_{\Omega^c_1}(\mathbb{C}:\nabla \mathbf{w}):\nabla\overline{\bphi}\,\mathrm d\mathrm \bx
+\int_{\Gamma}\widetilde{\bg}\cdot\overline{\bphi}\,\mathrm d\sigma\label{eq:v-form}\\
&=\int_{\Sigma_N}\mathcal{T}_{\bnu}(\mathbf{u_{\widetilde{\bff}}})\cdot \overline{\bphi}\,\mathrm d\sigma-\int_{\Omega^c_1}(\mathbb{C}:\nabla \mathbf{\bu_{\widetilde{\bff}}}):\nabla\overline{\bphi}\,\mathrm d\mathrm \bx,\qquad \forall\,\,\overline{\bphi}\,\,\in H_{\Sigma_D,\Sigma_N}^1(\Omega)^n. \nonumber
\end{align}
If $\bw$ is a solution of \eqref{eq:v-form}, then it is easy to show by choosing sufficiently smooth test functions that $\bu_1:=\bw\big|_{\Omega_1}$ and $\bu_2:=\bw\big|_{\Omega^c_1}+\bu_{\,\widetilde{\bff}}$ satisfy Problem~\eqref{eq:elast1'}. Conversely, multiplying these equations in~\eqref{eq:elast1'} by a test function and using transmission conditions, denote $\bw:=\bu_1$ in $\Omega_1$ and $\bw:=\bu_2-\bu_{\widetilde{\bff}}$ in $\Omega^c_1$, we can directly show that $\bw$ belongs to $  H_{\Sigma_D,\Sigma_N}^1(\Omega)^n$ and fulfils \eqref{eq:v-form}, where $(\bu_1,\bu_2)$ is a solution to Problem~\eqref{eq:elast1'}.

Let $a(\cdot, \cdot)$ and $\mathcal{F}(\cdot)$ be defined as follows:
\begin{align*}
a(\bw,\bphi)&:=\int_{\Omega_1}(\mathbb{C}:\nabla \mathbf{w}):\nabla\overline{\bphi}\,\mathrm d\mathrm \bx+\int_{\Omega^c_1}(\mathbb{C}:\nabla \mathbf{w}):\nabla\overline{\bphi}\,\mathrm d\mathrm \bx,\\
\mathcal{F}(\bphi)&:=-\int_{\Gamma}\widetilde{\bg}\cdot\overline{\bphi}\,\mathrm d\sigma+\int_{\Sigma_N}\mathcal{T}_{\bnu}(\mathbf{u_{\,\widetilde{\bff}}})\cdot \overline{\bphi}\,\mathrm d\sigma-\int_{\Omega^c_1}(\mathbb{C}:\nabla \mathbf{\bu_{\,\widetilde{\bff}}}):\nabla\overline{\bphi}\,\mathrm d\mathrm \bx.
\end{align*}
Hence, we can rewrite \eqref{eq:v-form} as the problem of finding $\bw\in H^1_{\Sigma_D,\Sigma_N}(\Omega)^n$ such that
$$
a(\bw,\bphi)=\mathcal{F}(\bphi)\quad\quad \mbox{for all}\quad\overline{\bphi}\in H_{\Sigma_D,\Sigma_N}^1(\Omega)^n.
$$

It is easy to see that $\mathcal{F}$ is a linear continuous operator from $H^1_{\Sigma_D,\Sigma_N}(\Omega)^n$ to $\mathbb{C}$. Combining Korn's inequality (cf.\cite{Mclean2010}) with the uniform strong convexity of $\mathbb{C}$, directly leads to the fact that the bilinear form $a(\cdot, \cdot)$ is strictly coercive and bounded. From the Lax-Milgram Lemma, there exists a unique solution $\bw\in H_{\Sigma_D,\Sigma_N}^1(\Omega)^n$ of \eqref{eq:v-form} such that
\begin{align*}
\|\bw\|_{H_{\Sigma_D,\Sigma_N}^1(\Omega)^n}&\leq\Big\|\bw\big|_{\Omega_1}\Big\|_{H^1(\Omega_1)^n}+\Big\|\bw\big|_{\Omega^c_1}\Big\|_{H^1(\Omega^c_1)^n}\leq \|\mathcal{F}\|
 \leq C\big(\|\widetilde{\bff}\|_{H^{\frac{1}{2}}(\Gamma)^n}+\|\widetilde{\bg}\|_{H^{-\frac{1}{2}}(\Gamma)^n}\big)\\
 &\leq C\big(\|\bff\|_{H_{00}^{\frac{1}{2}}(\mathcal{S})^n}+\|\bg\|_{H^{-\frac{1}{2}}_0(\mathcal{S})^n}\big).
 \end{align*}
The above results imply immediately that there exists a unique solution $\bu\in H_{\Sigma_D,\Sigma_N}^1(\Omega\backslash\overline{\mathcal{S}})^n$ to Problem~\eqref{eq:elast1} and $\bu$ can be estimated by $\bff$ and $\bg$ with respect to $H_{00}^{\frac{1}{2}}(\mathcal{S})^n$ norm and $H^{-\frac{1}{2}}_0(\mathcal{S})^n$ norm, respectively.

The proof is complete.
\end{proof}

\section{The inverse problem and main results}\label{sec:1}
In this section, we are devoted to studying the uniqueness results of the inverse dislocation problem, which is composed of identifying the dislocations $\mathcal{S}$ and jump functions $\bff$ and $\bg$ over $\mathcal{S}$ by observation data on an open set $\Sigma_0\subset\Sigma_N$. For our study, we formulate the inverse problem as
\beq\nonumber
\Lambda_{\mathcal{S};\bff,\bg}=\bu\big|_{\Sigma_0},
\eeq
where $\bu$ is the solution to Problem~\eqref{eq:elast1}. That is, $\Lambda_{\mathcal{S};\bff,\bg}$ contains the elastic deformation data caused by the dislocation $(\mathcal{S};\bff,\bg)$ and observed on $\Sigma_0\subset \Sigma_N$. The inverse problem we are devoted to can be formulated by
\beq\label{eq:IP2}
\Lambda_{\mathcal{S};\bff,\bg}(\bu\big|_{\Sigma_0,\Sigma_N})\rightarrow\mathcal{S},\bff,\bg.
\eeq

The uniqueness results can be proved under some assumptions about the geometry of the dislocation $\mathcal{S}$ and a {\it priori} information about Lam\'e parameters $\lambda(\bx)$ and $\mu(\bx)$ of the elastic solid $\Omega$, where $\lambda$ and $\mu$ are real constants and satisfy the strong convexity condition \eqref{eq:tensor2}.

To describe the geometry of $\mathcal{S}$, we introduce some notations for the geometric setup; see Fig.~\ref{fig:2} for a schematic illustration. Given $\bx_c\in \mathbb{R}^2$ and constants $\theta_m,\,\theta_M\in(-\pi,\pi)$ such that $\theta_M-\theta_m\in (0,\pi)$, we consider the following open sector
\begin{align}\label{geo:1}
\mathcal{K}_{\bx_c}\,=\left\{ \bx\in \mathbb{R}^2\big| \,{\bf 0}<\bx=\bx_c+(r\cos\theta,r\sin\theta)^\top, \,\, \theta_{m}<\theta<\theta_{M},r>0 \right\}
\end{align}
with boundaries
\begin{align*}
	\Gamma^+_{\bx_c}&=\left\{\bx\in \mathbb{R}^2\big|\,\bx=\bx_c+(r\cos\theta_M,r\sin\theta_M)^\top,\, r>0  \right\},\\
	\Gamma^-_{\bx_c}&=\left\{\bx\in \mathbb{R}^2\big|\, \bx=\bx_c+(r\cos\theta_m,r\sin\theta_m)^\top,\, r>0  \right\}.
\end{align*}
The point $\bx_c$ is said to be a planar corner point (also named 2D corner point) with opening angle $\theta_M-\theta_m$ and boundaries $\Gamma^\pm_{\bx_c}$.
Let
\begin{alignat}{2}
\mathcal{C}_{\bx_c,h}&:=\mathcal{K}_{\bx_c}\cap B_h(\bx_c),\quad& \Gamma^\pm_{\bx_c,h}&:=\Gamma^\pm_{\bx_c}\cap B_h(\bx_c),\nonumber\\
\Lambda^{\bx_c}_{h}&:=\mathcal{K}_{\bx_0}\cap \partial B_h(\bx_c),\quad & \Sigma_{\bx_c}&:=\mathcal{C}_{\bx_c,h}\backslash \mathcal{C}_{\bx_c,h/2},\label{eq:ball1}
\end{alignat}
where $B_h(\bx_c)$ denotes an open disk centered at $\bx_c$ of radius $h\in \mathbb{R}_+$.
For the sake of brevity, we use $B_{h}, \,\mathcal{K},\, \Gamma^\pm, \,\mathcal{C}_{h},\,\Gamma^\pm_{h},\,\Lambda_{h}$ and $\Sigma$ to represent the corresponding notations at the origin.

\begin{figure}[ht]
\centering
\subfigure{\includegraphics[width=0.40\textwidth]{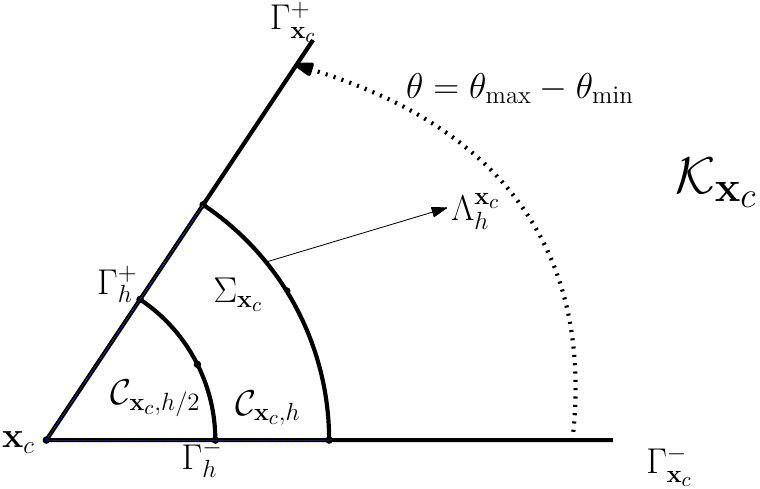}}\qquad
\subfigure{\includegraphics[width=0.44\textwidth]{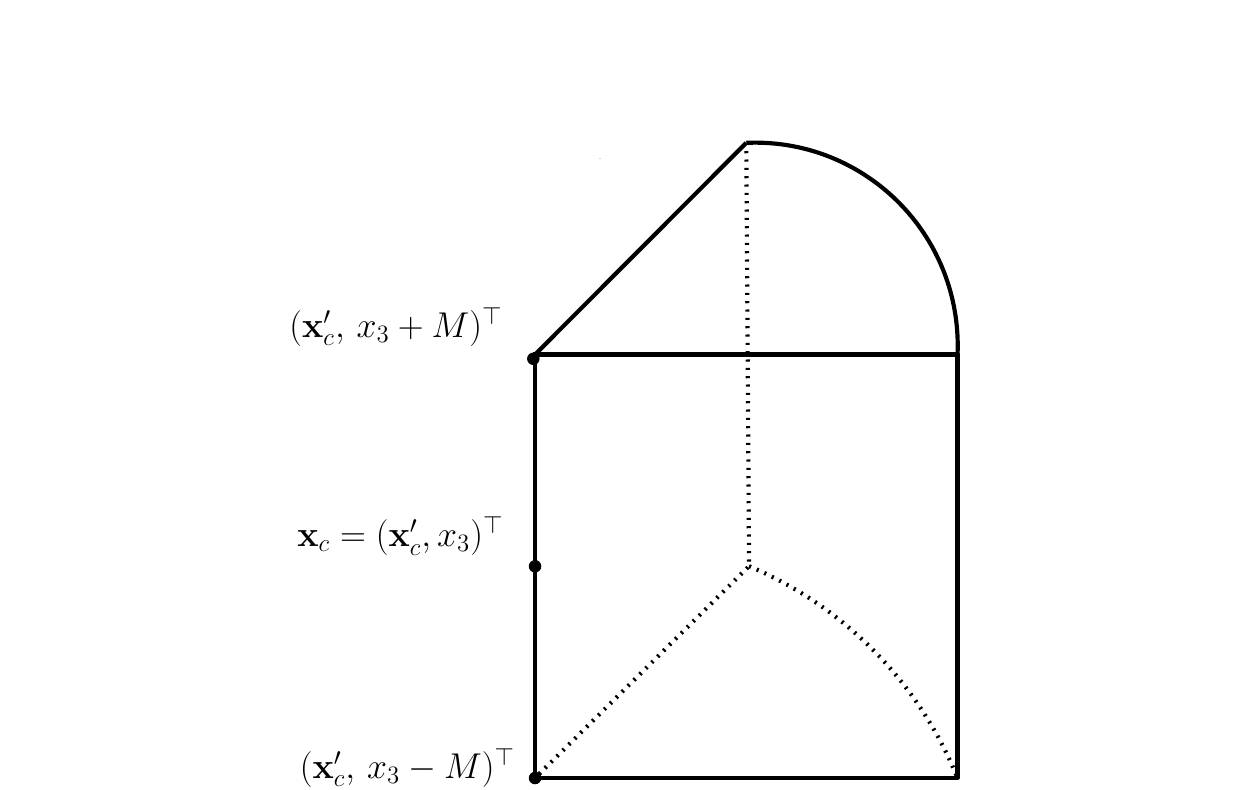}}\\
\caption{Schematic illustration of a 2D corner/3D edge corner. \label{fig:2}}
\end{figure}

\subsection{Main uniqueness results}

Before giving the unique results in Theorem~\ref{th:main_loca}-Theorem \ref{th:main2}, we introduce some admissible conditions about the dislocation for our subsequent study.

\begin{defn}\label{def:Admis2}
Let $\Omega$ be a bounded Lipschitz domain in $\mathbb{R}^n (n=2,3)$.
We say that $(\mathcal{S};\bff,\bg)$ belongs to the admissible class $\mathcal{T}$ if the following conditions are fulfilled:
\begin{itemize}
\item[{\rm (1)}] In $\mathbb R^2$, $\mathcal{S}\subset \mathbb{R}^{2}$ is an oriented  Lipschitz curve. There exists at least one planar corner point $\bx_c$ on $\mathcal{S}$ such that $\Gamma^\pm_{\bx_c,h}\subset \mathcal{S}$, where  $\Gamma^\pm_{\bx_c,h}=\partial \mathcal{K}_{\bx_0}\cap B_h(\bx_c)$ and $\mathcal{C}_{\bx_c,h}=B_h(\bx_c)\cap \mathcal{K}_{\bx_c}=B_h(\bx_c)\cap \Omega_1$. Here, $\mathcal{K}_{\bx_c}$ and $B_h(\bx_c)$ are defined in \eqref{geo:1} and $\Omega_1= {\mathrm {enclose}}(\mathcal{S})$ is given in \eqref{eq:omega1}.
\item[{\rm (2)}]In $\mathbb R^3$, $\mathcal{S}\subset \mathbb{R}^{3}$ is an oriented  Lipschitz surface. Suppose that $\mathcal{S}$ possesses at least one 3D edge corner point $\bx_c=(\bx'_c,x^3_c)^\top\in\mathbb{R}^3$, where $\bx'_c\in \mathbb{R}^2$ is a planner corner point. In other words, for sufficient small positive numbers $h$ and $M$,  we have that $\Gamma^\pm_{\bx'_c,h}\times (-M,M)\subset \mathcal{S}$ and $B'_h(\bx'_c)\times (-M,M)\cap\Omega_1=\mathcal{C}'_{\bx_c,h}\times (-M,M)$, where $\Gamma^\pm_{\bx'_c,h}$ are two edges of a sectorial corner at $\bx'_c$ and $B'_h(\bx'_c)$  is an open disk centered at $\bx'_c$ with  radius $h$, which are  defined in \eqref{eq:ball1}. The opening angle of the sectorial corner at $\bx'_c$ is referred to be the opening angle of the corresponding 3D edge corner.
\item[{\rm (3)}] In $\mathbb R^2$, let $\bff_j:=\bff\big|_{\Gamma^j_{\bx_c,h}}$ and $\bg_j:=\bg\big|_{\Gamma^j_{\bx_c,h}}$ satisfy $\bff_j \in C^{1,\alpha_j}(\Gamma^j_{\bx_c,h})^2$ and  $\bg_j\in C^{\beta_j}(\Gamma^j_{\bx_c,h})^2$ with $\alpha_j,\beta_j$ being in $(0,1)$ and $j=+,-$, where $\bx_c$ and $\Gamma^\pm_{\bx_c,h}$ are just the ones in $(1)$.
\item[{\rm (4)}] In $\mathbb R^3$, let $\bff_j:=\bff\big|_{\Gamma'^j_{\bx'_c,h}\times (-M,M)}$ and $\bg_j=\bg\big|_{\Gamma'^j_{\bx'_c,h}\times (-M,M)}$ fulfill that $\bff_j \in C^{1,\alpha_j}(\Gamma'^j_{\bx'_c,h}\times (-M,M))^3$ and  $\bg_j\in C^{\beta_j}(\Gamma'^j_{\bx'_c,h}\times (-M,M))^3$  with $\alpha_j,\beta_j$ being in $(0,1)$ and $j=+,-$, where $\bx'_c$ and $\Gamma'^\pm_{\bx'_c,h}$ are just the ones in $(2)$. Furthermore, $\bff_j$ and $\bg_j$ are are independent of $x_3$.
\item[{\rm(5)}] Let  $\mathcal{V}_{\mathcal{S}}$ signify the set of 2D corners/3D edge corners of $\mathcal{S}$. In $\mathbb R^3$, we denote $\bg=(\bg^{(1,2)}, g^3)$, where  $\bg^{(1,2)}$ and $g^3$ represent the first two components and the third component of $\bg$, respectively. Either the following assumptions $\mathcal{A}_1$ or $\mathcal{A}_2$ is satisfied,
\begin{align*}
 &\qquad\quad {\it\mbox{ Assumption} \,\,\mathcal{A}_{\rm 1}}: \,\,\mbox{For any }\,\,\bx_c\in \mathcal{V}_{\mathcal{S}}, \quad \bff_-(\bx_c)\neq\bff_+(\bx_c),\\[2mm]
&\qquad\quad  {\it\mbox{ Assumption} \,\,\mathcal{A}_{\rm 2}}: \mbox{If there exists a point}\, \bx_c\in \mathcal{V}_{\mathcal{S}}\,\, \mbox{such that }  \bff_-(\bx_c)=\bff_+(\bx_c),\\[2mm]
&\qquad \qquad\qquad   \mbox{then}\,\, \nabla\bff_+(\bx_c)= \nabla\bff_-(\bx_c)=\bf 0\,\,\mbox{and}\,
 \left\{\begin{array}{ll}
	\bg_+(\bx_c)\neq W_{\bx_c}\bg_-(\bx_c),\qquad\mbox{ if $n=2$},\\[5mm]
\bg^{(1,2)}_+(\bx_c)\neq W_{\bx'_c}\bg^{(1,2)}_-(\bx_c),\mbox{ if $n=3$},\\
\mbox{or}\quad g^3_+(\bx_c)\neq 0,\\
\mbox{or} \quad g^3_-(\bx_c)\neq 0,
\end{array}\right.
\end{align*}
where  
\beq\nonumber 
\qquad\quad  \bg_\pm=(\bg^{(1,2)}_\pm,\, g^3_\pm)^\top,\,
W_{\bx_c}=
		\begin{bmatrix}-\cos\theta_{\bx_c}&-\sin\theta_{\bx_c}\\
    -\sin\theta_{\bx_c} &\cos\theta_{\bx_c}
    \end{bmatrix},\,
W_{\bx'_c}=\begin{bmatrix}-\cos\theta_{\bx'_c}&-\sin\theta_{\bx'_c}\\
    -\sin\theta_{\bx'_c} &\cos\theta_{\bx'_c}
    \end{bmatrix}.
\eeq
Here, $\theta_{\bx_c}$ and $\theta_{\bx'_c}$ denote the opening angle at the 2D corner/3D edge corner  point $\bx_c$ of $\mathcal{S}$.
\end{itemize}
\end{defn}

\begin{rem}
The admissible conditions that $\mathcal{S}$ possesses at least one planar corner/3D edge corner in Definitions \ref{def:Admis2} can be easily fulfilled in generic physical scenarios. For example, $\mathcal{S}$ is a piecewise linear fault curve/surface. In what follows, $(\mathcal{S};\bff,\bg)$ is said to be an admissible dislocation with jump vectors $\bff$ and $\bg$ if it fulfills the conditions in Definitions~\ref{def:Admis2}.
\end{rem}

The uniqueness results in \cite{AspriBerettaMazzucato2022, AspriBerettaMazzucatoHoop2020} for determining $\mathcal S$ only focused on the case that $\mathcal{S}$ is open and there is only the displacement discontinuity on the fault curve/surface. Those methodology developed in \cite{AspriBerettaMazzucato2022, AspriBerettaMazzucatoHoop2020} cannot deal with the case that  $\mathcal{S}$ is closed. In our study, we can handle these two situations and also allow the occurrence of discontinuities in both displacement and traction fields on the dislocation. In Theorem \ref{th:main_loca}, we obtain a local uniqueness result for the inverse dislocation problem~\eqref{eq:IP2} by using a single displacement measurement on $\Sigma_0$ whose proof is postponed in Section \ref{se:proof}.

\begin{thm}\label{th:main_loca}
Let $(\mathcal{S}_1; \bff^1,\bg^1)$ and $(\mathcal{S}_2; \bff^2,\bg^2)$ belong to $\mathcal{T}$. Assume that $\mathrm{supp}(\bg^i)=\mathrm{supp}(\bff^i)=\overline{\mathcal{S}}_i$ and $\bu_i$ is the unique solution to Problem \eqref{eq:elast1} in $H^1_{\Sigma_D,\Sigma_N}(\Omega\backslash\overline{\mathcal{S}_i})$ with respect with to $\bg=\bg^i$, $\bff=\bff^i$ and $\mathcal{S}=\mathcal{S}_i$,  respectively, for $i=1,2$. If $\bu_1\big|_{\Sigma_0}=\bu_2\big|_{\Sigma_0}$, then   $\mathcal{S}_1\Delta\mathcal{S}_2:=(\mathcal{S}_1\backslash\mathcal{S}_2)\cup(\mathcal{S}_2\backslash\mathcal{S}_1)$  cannot contain a planar corner/3D edge corner $\bx_c$.
\end{thm}

In Theorem \ref{th:main1} and \ref{th:main2} we derive the global uniqueness results for the inverse dislocation problem~\eqref{eq:IP2} on the determination of the dislocation surface $\mathcal{S}$ and $\bff$, $\bg$ from a given single displacement measurement on $\Sigma_0$. The proofs of Theorems \ref{th:main1} and \ref{th:main2} are postponed in Section \ref{se:proof}.

We first consider the case that $\mathcal{S}_1$ and $\mathcal{S}_2$ are closed.
\begin{thm}\label{th:main1}
Let $(\mathcal{S}_1; \bff^1,\bg^1)$ and $(\mathcal{S}_2; \bff^2,\bg^2)$ belong to $\mathcal{T}$, where $\mathcal{S}_1$ and $\mathcal{S}_2$ are closed. Assume that $\Omega_1= {\mathrm {enclose}}(\mathcal{S}_1)$ and $\Omega_2= {\mathrm {enclose}}(\mathcal{S}_2)$ are two convex polygons in $\mathbb R^2$ or two convex polyhedrons in $\mathbb R^3$, where $\mathcal{S}_i=\bigcup_{k=1}^{m_i} \Pi_{i,k}$ $(i=1,2)$. Here $\Pi_{i,k}$ is the $k$-th edge or surface of the polygon or polyhedron $\Omega_i$.    Let $\bu_i$ be the unique solution to Problem \eqref{eq:elast1} in $H^1_{\Sigma_D,\Sigma_N}(\Omega\backslash\overline{\mathcal{S}})$ with respect to $\bg^i$ and $\bff^i$, respectively, for $i=1,2$. If $\bu_1\big|_{\Sigma_0}=\bu_2\big|_{\Sigma_0}$, then $\mathcal{S}_1=\mathcal{S}_2$, namely $m_1=m_2:=m,\  \Pi_{1,k}=\Pi_{2,k}:=\Pi_{k}\ (k=1,\ldots,m)$.
Furthermore, assume that $\bff^i$ and $\bg^i$ are piecewise-constant functions on $ \Pi_k$. Then we have
\beq\label{eq:unque1}
   (\bff^1-\bff^2)\big|_{\Pi_{k+1}}=(\bff^1-\bff^2)\big|_{\Pi_{k}},\,\, (\bff^1-\bff^2)\big|_{\Pi_1}=(\bff^1-\bff^2)\big|_{\Pi_m}, \quad k=1,\cdots,m-1.
   \eeq
   
Furthermore, if $\nabla \bff^1(\bx_k)=\nabla \bff^2(\bx_k)$ at the 2D corners/3D edge corners $\bx_k$, $k=1,2,\cdots,m$, then we have 
    \beq\label{eq:unque1'}
      (\bg^1-\bg^2)\big|_{\Pi_{k+1}}=W_{\bx_k}(\bg^1-\bg^2)\big|_{\Pi_{k}}, \quad(\bg^1-\bg^2)\big|_{\Pi_1}=W_{\bx_m}(\bg^1-\bg^2)\big|_{\Pi_m},
    \eeq
   where $W_{\bx_k}=\begin{bmatrix}-\cos\theta_{\bx_k}&-\sin\theta_{\bx_k}\\
    -\sin\theta_{\bx_k} &\cos\theta_{\bx_k}
    \end{bmatrix}$ is similarly defined as $W_{\bx_c}$ in Definition~\ref{def:Admis2} and $\theta_{\bx_k}$ corresponds to the opening angle at $\bx_k$.
\end{thm}

In Theorem \ref{th:main2}, we investigate the unique determination of a piecewise curve or a piecewise surface $\mathcal{S}$, where $\mathcal{S}$ is open. Before that, we introduce the corresponding definition.
\begin{defn}\label{def:poly}
	Suppose that $\mathcal{S}\subset \mathbb R^n$ $(n=2,3)$ is open. Under rigid motion, we can have a suitable and fixed but arbitrary coordinate system such that $\mathcal{S}\subset \mathbb R^2$ is the graph of a function $f(x_1)$, where $x_1\in [a,b]$. If $[a,b]=\cup_{i=1}^\ell [a_i,a_{i+1}]$ with $l\geq 3$, $a_i<a_{i+1}$, $a_1=a$ and $a_\ell=b$, which is piecewise linear polynomial on each piece $[a_i,a_{i+1}]$, then $\mathcal{S}\subset \mathbb R^2$ is referred as a linear piecewise  curve. Similarly, under rigid motion, we can have a suitable and fixed but arbitrary coordinate system such that $\mathcal{S}\subset \mathbb R^3$ is the graph of a function $f(x_1,x_3)$, where $(x_1,x_3)\in [a_1,a_2] \times [b_1,b_2]$. If $f(x_1,c)$ with $c\in [b_1,b_2]$ being fixed satisfies $f(x_1,c)=g(x_1)$, where the graph of $g(x_1)$ is a piecewise curve like the ones in $\mathbb R^2$, then $\mathcal{S}\subset \mathbb R^3$ is referred as a piecewise surface. See Fig.~\ref{fig:3} for a schematic illustration.
\end{defn}
\begin{figure}[ht]
\centering
\subfigure{\includegraphics[width=0.50\textwidth]{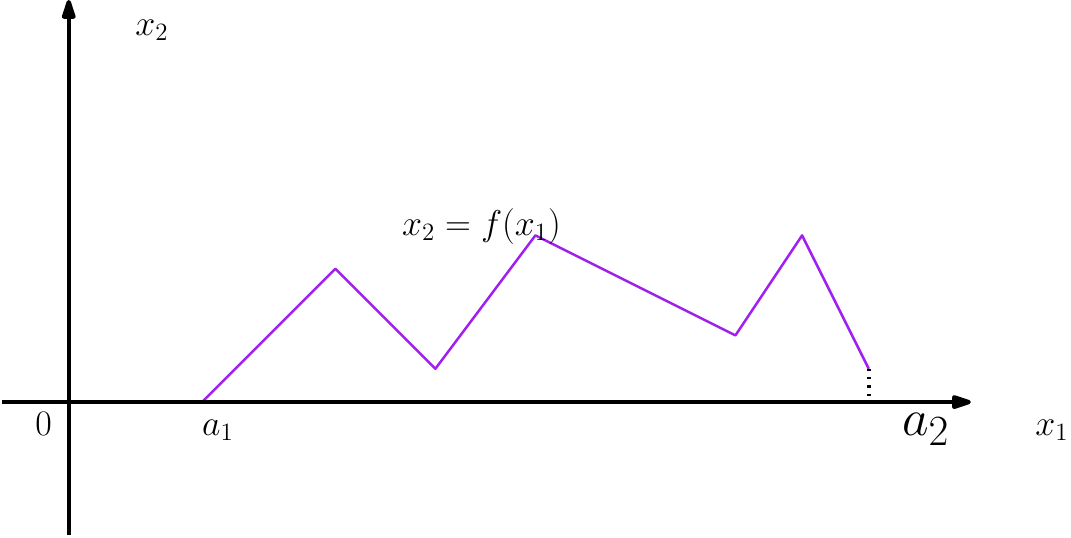}}\qquad
\subfigure{\includegraphics[width=0.40\textwidth]{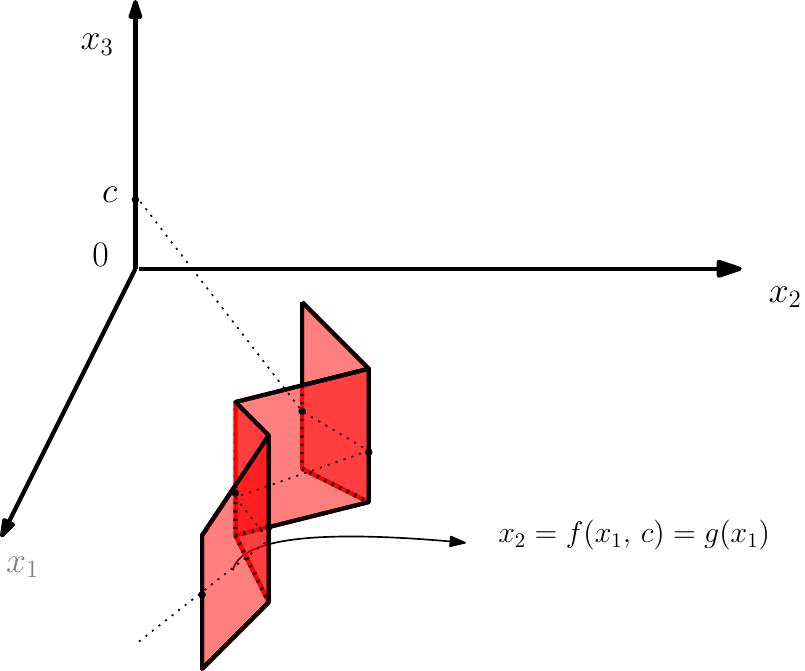}}\\
\caption{Schematic illustration of a piecewise curve or a piecewise surface.  \label{fig:3}}
\end{figure}

\begin{rem}
Indeed, Definition \ref{def:poly} provides an assumption regarding dislocations $\mathcal{S}$ as piecewise linear polynomial curves/piecewise surfaces, a specific class of graphs with 2D corners/3D edge corners. This assumption coincides with that made by Aspri et al. in \cite{AspriBerettaMazzucato2022, AspriBerettaMazzucato20222}, where $\mathcal{S}$ is assumed to be a Lipschitz graph with respect to a fixed but arbitrary coordinate system. However, we restrict ourselves to this special class of graphs. In particular, the definition of piecewise surfaces shall be combined with the dimension reduction operator $\mathcal{P}$ in Subsection \ref{subse:case2} to study the relevant continuities about jump functions at the 3D edge corners.
\end{rem}

\begin{thm}\label{th:main2}
Let $(\mathcal{S}_1; \bff^1,\bg^1)$ and $(\mathcal{S}_2; \bff^2,\bg^2)$ belong to $\mathcal{T}$, where $\mathcal{S}_1$ and $\mathcal{S}_2$ are open. Suppose that the curves/surfaces $\mathcal{S}_1$ and $\mathcal{S}_2$ are piecewise curves in $\mathbb R^2$ or piecewise surfaces in $\mathbb R^3$ (as defined in  Definition \ref{def:poly}). Let $\bu_i$ be the unique solution to Problem \eqref{eq:elast1} in $H^1_{\Sigma_D,\Sigma_N}(\Omega\backslash\overline{\mathcal{S}_i})$ with respect to $\bg^i$, $\bff^i$ and $\mathcal{S}_i$, where $\bff^i\in H^{\frac{1}{2}}_{00}(\mathcal{S}_i)$ and $\bg^i\in H^{-\frac{1}{2}}_0(\mathcal{S}_i)$ with $\mathrm{supp}(\bg^i)=\mathrm{supp}(\bff^i)=\overline{\mathcal{S}}_i$, $i=1,2$. If
\begin{align}\label{eq:thm33}
	\bu_1\big|_{\Sigma_0}=\bu_2\big|_{\Sigma_0},
\end{align}
 then \beq\nonumber
\mathcal{S}_1=\mathcal{S}_2,\quad \bff^1=\bff^2\,\,\,\mbox{and}\,\,\,\bg^1=\bg^2.
\eeq
\end{thm}

\section{Local results of slips $\bff$ and $\bg$ at the corners on $\mathcal{S}$}\label{se:pre}
In this section, we shall derive several auxiliary propositions that describe the local results of jump vectors $\bff$ and $\bg$ at the corner points of $\mathcal{S}$. These auxiliary results play a key role in establishing our main results in Theorems \ref{th:main_loca}, \ref{th:main1}, and \ref{th:main2}. We shall introduce two kinds of the so-called CGO (complex geometrical optics) solutions that satisfy different Lam\'e /acoustic equations to derive these propositions. The CGO solutions will be later used as test functions to obtain these significant properties of the jump functions of the displacement and the traction field at a 2D corner/3D edge corner point, namely the results in \eqref{eq:local2} and \eqref{eq:vansh2}.

\subsection{Local results for 2D case}
This subsection analyzes the local behaviors of jump vectors $\bff$ and $\bg$ around a planar corner. Next, we introduce the first kind of CGO solution $\bu_0$  introduced in \cite{BlastenLin2019}, which is given  by
\begin{equation}\label{eq:lame3}
\bu_0( \bx )=
\begin{pmatrix}
    \exp(-s\sqrt{z})\\
    {\rm i}\exp(-s\sqrt{z})
    \end{pmatrix}:=\begin{pmatrix}
    u^0_1( \bx)\\
    u^0_2( \bx)
\end{pmatrix}  \quad\mbox{in}\quad \Omega,\quad \bx=(x_1,x_2)^\top,
\end{equation}
 where $z=x_{1}+{\rm i}\,x_{2}$, $s\in \mathbb{R}_{+}$ and $\Omega\cap(\mathbb{R}_{-}\cup\{ 0
\})=\emptyset$. Here the complex square root of $z$ is defined as
\begin{equation}\notag 
\sqrt{z}=\sqrt{|z|}\left(\cos\frac{\theta}{2}+{\rm i}\sin\frac{\theta}{2}\right ),
\end{equation}
where $-\pi<\theta<\pi$ is the argument of $z$. Furthermore, it yields that  $\mathcal{L} \,\bu_0={\bf 0}$ in $\Omega$.

Some significant properties and regularity results of the  CGO solution given in \eqref{eq:lame3} need to be reviewed, which is beneficial for the subsequent analysis.

\begin{lem} \cite[Proposition 3.1]{BlastenLin2019} \label{lem:cgo1}
Let $\bu_0$ be given as above.  Then we have the following properties
\beq\nonumber
\int_\mathcal{K} u^0_1(\bx){\rm d}\bx=6{\rm i}(e^{-2\theta_M{\rm i}}-e^{-2\theta_m \mathrm i})s^{-4}
\eeq
and
\beq\label{eq:lame6}
\int_{\mathcal{K}}|u^0_{j}(\mathbf{x})||\mathbf{x}|^{\alpha}{\rm d}\mathbf{x}\leq\frac{2(\theta_M-\theta_m)\Gamma(2\alpha+4)}{\delta_{\mathcal{K}}^{2\alpha+4}}s^{-2\alpha-4},\quad j=1,2,
\eeq
where $\mathcal{K}$ is defined in Section~\ref{sec:1}, $\alpha,h>0$, $\delta_{\mathcal{K}}=\min\limits_{\theta_m<\theta<\theta_M} \cos \frac{\theta}{2}$ is a positive constant.
\end{lem}

The following critical estimate can be obtained by using the Laplace transform and the exponential function with negative power.

\begin{lem}\label{10-lem:u0 int1} For any $\alpha>0$, if $\omega(\theta)>0$, then we have 
$$
\lim\limits_{s\rightarrow +\infty}\int^h_0 r^{\alpha}e^{-s\sqrt{r}\omega(\theta)}\mathrm{d}\mathrm{r}=\Oh(s^{-2\alpha-2}).
$$
\end{lem}

We next recall some critical lemmas about the regularity of the CGO solution $\bu_0$ defined in \eqref{eq:lame3}. \begin{lem}\label{10-lem:23}\cite[Lemma~2.3]{DiaoLiuSun2021}
	Let $\mathcal{C}_{\bx_c,h}$ be defined in \eqref{eq:ball1} and $\bu_0$ be given in \eqref{eq:lame3}. Then $\bu_0 \in H^1(\mathcal{C}_{\bx_c,h})^2$ and $\mathcal{ L} \, \bu_0 =\mathbf 0$ in $\mathcal{C}_{\bx_c,h}$. Furthermore, it holds that
	\begin{equation*}
	\big\|\bu_0\big\|_{L^2(\mathcal{C}_{\bx_c,h})^2 } \leq 	 \sqrt{\theta_M-\theta_m} e^{- s\sqrt{\Theta }\, h }
	\end{equation*}
	and
	\begin{equation*}
		\Big  \||\mathbf x|^\alpha \mathbf  u_0 \Big \|_{L^{2}(\mathcal{C}_{\bx_c,h} )^2  }\leq s^{-2(\alpha+1 )} \frac{2\sqrt{(\theta_M-\theta_m)\Gamma(4\alpha+4) }   }{(2\delta_{\mathcal{K}})^{2\alpha+2  } } ,
	\end{equation*}
where $ \Theta  \in [0,h ]$ and $\delta_{\mathcal{K}}$ is defined in \eqref{eq:lame6}.
\end{lem}

\begin{lem}\cite[Lemma 2.8]{DiaoCaoLiu2021}\label{10-lem:u0 int}
	Let $\Gamma^\pm_h$ and $u^0_1(\mathbf x)$ be respectively defined in \eqref{eq:ball1} and \eqref{eq:lame3} with $\bx_c$ coinciding with the origin. We have	
\begin{align*}
	\begin{split}
\int_{\Gamma^+_{h} }  u^0_1 (\mathbf x)  {\rm d} \sigma &=2 s^{-2}\left( \hat{\mu}(\theta_M )^{-2}-   \hat{\mu}(\theta_M )^{-2} e^{ -s\sqrt{h} \,\hat{\mu}(\theta_M ) }\right.  \\&\left. \hspace{3.5cm} -  \hat{\mu}(\theta_M )^{-1} s\sqrt{h}   e^{ -s\sqrt{h}\, \hat{\mu}(\theta_M ) }  \right  ),  \\
\int_{\Gamma^-_{h}  }  u^0_1 (\mathbf x)  {\rm d} \sigma &=2 s^{-2} \left( \hat{\mu}(\theta_m )^{-2}-   \hat{\mu}(\theta_m )^{-2} e^{ -s\sqrt{h}\mu(\theta_m )} \right.  \\&\left. \hspace{3.5cm}  -  \hat{\mu}(\theta_m )^{-1} s\sqrt{h}   e^{ -s\sqrt{h}\,\hat{\mu}(\theta_m ) }  \right  ),	
	\end{split}
\end{align*}
	where $ \hat{\mu}(\theta ):=\cos(\theta/2) +\mathrm i \sin( \theta/2 )=e^{\mathrm{i}\theta/2}$.
\end{lem}

To prove our main Theorems~\ref{th:main_loca}, \ref{th:main1} and \ref{th:main2}, the next critical auxiliary proposition is needed. Since $\mathcal L$ is invariant under rigid motion, in what follows, the underlying corner point $\mathbf x_c$ coincides with the origin.

\begin{prop}\label{prop:trans2}
Under the same setup about $\mathcal{C}_h$ and $\Gamma^\pm_h$ given in \eqref{eq:ball1} with $\bx_c$ coinciding with the origin. Let $\bv\in H^1(\mathcal{C}_h)^2$ and $\bw\in H^1(\mathcal{C}_h)^2$ satisfy
\allowdisplaybreaks
\begin{align}\label{eq:trans1}
\begin{cases}
\mathcal{L}\,\bv=\mathbf{0},\qquad \mathcal{L}\,\bw=\mathbf{0}  &\mbox{in}\quad\mathcal{C}_h,\\
\bv-\bw=\bff_+,\,\mathcal{T}_{\bnu}\,\bv-\mathcal{T}_{\bnu}\,\bw=\bg_+ &\mbox{on}\quad \Gamma^+_h,\\
\bv-\bw=\bff_-,\,\mathcal{T}_{\bnu}\,\bv-\mathcal{T}_{\bnu}\,\bw=\bg_- &\mbox{on}\quad \Gamma^-_h\\
\end{cases}
\end{align}
 with $\bff_j\in H^{\frac{1}{2}}(\Gamma_h^j)^2\cap C^{1,\alpha_j}(\Gamma_h^j)^2$ and $\bg_j\in H^{-\frac{1}{2}}(\Gamma_h^j)^2\cap C^{\beta_j}(\Gamma_h^j)^2$, where $j=+,-$ and $\alpha_+,\alpha_-,\beta_+,\beta_- \in (0,1)$.
Then we have 
 \begin{align}\label{eq:local2}
\bff_+({\bf 0})=\bff_-({\bf 0}).
\end{align}
Furthermore, if $\nabla\bff_j({\bf 0})={\bf 0}$ holds, then it implies that 
\begin{align}\label{eq:local2'}
\bg_+({\bf 0})=W\,\bg_-({\bf 0}),
\end{align}
where $W=\begin{bmatrix}-\cos(\theta_M-\theta_m),&-\sin(\theta_M-\theta_m)\\
    -\sin(\theta_M-\theta_m), &+\cos(\theta_M-\theta_m)
    \end{bmatrix}$ is an orthogonal matrix.
\end{prop}

\begin{proof}
Let $\Re\bv$ and $\Im\bv$ denote the real and imaginary parts of $\bv$, respectively, while $\Re\bw$ and $\Im\bw$ represent the real and imaginary parts of $\bw$. Hence, we obtain the following splittings: 
\begin{align*}\bu=\Re\bu+\mathrm{i}\Im\bu\quad \mbox{and}\quad \bw=\Re\bw+\mathrm{i}\Im\bw.
\end{align*}
It is straightforward to verify that both $(\Re\bv,\Re\bw)$ and $(\Im\bv,\Im\bw)$ satisfy \eqref{eq:trans1}.
Considering the symmetric role of $(\Re\bv,\Re\bw)$ and $(\Im\bv,\Im\bw)$, we only need to demonstrate that the corresponding results hold for $(\Re\bv,\Re\bw)$.
By applying a similar proof process, we can establish that these results remain valid for $(\Im\bv,\Im\bw)$ and consequently for $(\bv,\bw)$. Due to Betti's second formula, we have the following integral identity
\begin{align}
&\int_{\Gamma^+_h}\Re\bg_+\cdot\bu_0-\mathcal{T}_{\bnu}\bu_0\cdot\Re\bff_+\,\mathrm d\sigma+\int_{\Gamma^-_h}\Re\bg_-\cdot\bu_0-\mathcal{T}_{\bnu}\bu_0\cdot\Re\bff_-\,\mathrm d\sigma\nonumber\\
&=\int_{\Lambda_h}\mathcal{T}_{\bnu}\big(\Re\bv-\Re\bw\big)\cdot\bu_0-\mathcal{T}_{\bnu}\bu_0\cdot\big(\Re\bv-\Re\bw\big)\,\mathrm d\sigma.\label{eq:iden2}
\end{align}
Since $\bff_j\in C^{1,\alpha_j}(\Gamma_h^j)^2$ and $\bg_j\in C^{\beta_j}(\Gamma_h^j)^2 $ for $j=+,-$, we have the expansions as follows

\begin{align}
 \bff_j(\bx)&=\bff_j({\bf0})+\nabla \bff_j({\bf 0})\cdot \bx+ \delta\bff_j(\bx),\,\,\quad \big| \delta\bff_j(\bx) \big|\leq A_j|\bx|^{1+\alpha_j},\label{eq:hollder5} \\
 \bg_j(\bx)&=\bg_j({\bf0})+ \delta\bg_j(\bx),\quad \big| \delta\bg_j (\bx)\big|\leq B_j|\bx|^{\beta_j},  \nonumber
 \end{align}
where $A_j$ and $B_j$ are positive.
From the expression of $\bu_0$, it is easy to imply that
\begin{align*}
\frac{\partial u^0_1}{\partial r}=-\frac{s}{2\sqrt{r}}e^{-s r^{1/2}\hat{\mu}(\theta)+\mathrm{i}\frac{\theta}{2}}\quad \mbox{and}\quad
\frac{\partial u^0_1}{\partial \theta}=-\frac{\mathrm{i}s \sqrt{r}}{2}e^{-s r^{1/2}\hat{\mu}(\theta)+\mathrm{i}\frac{\theta}{2}},
\end{align*}
where $\hat{\mu}(\cdot)$ is given by Lemma~\ref{10-lem:u0 int}.
Thus, we directly obtain
\begin{align*}
\frac{\partial u^0_1}{\partial x_1}=-\frac{s}{2\sqrt{r}}e^{-s r^{1/2}\hat{\mu}(\theta)-\mathrm{i}\frac{\theta}{2}}\quad \mbox{and}\quad
\frac{\partial u^0_1}{\partial x_2}=-\frac{\mathrm{i}s \sqrt{r}}{2}e^{-s r^{1/2}\hat{\mu}(\theta)-\mathrm{i}\frac{\theta}{2}}.
\end{align*}
Notice that $u^0_2(\bx)=\mathrm{i}\,u^0_1(\bx)$, we get
$$
\nabla\bu_0=-\frac{s}{2\sqrt{r}}e^{-s\sqrt{r}\hat{\mu}(\theta)-\frac{\theta}{2}\mathrm{i}}\begin{bmatrix}1&\mathrm{i}\\\mathrm{i}&-1\end{bmatrix},
$$
Therefore, one can prove that
\begin{align}\label{eq:esti2_1}
\int_{\Gamma^+_h}  \mathcal{T}_{\bnu_M}\bu_0  \,\mathrm d\sigma&=\int_{\Gamma_h^+} -\frac{s}{2r^{\frac{1}{2}}}e^{-s\sqrt{r}\hat{\mu}(\theta_M)-\mathrm{i}\frac{\theta_M}{2}}\begin{bmatrix}1&\mathrm{i}\\\mathrm{i}&-1\end{bmatrix}\cdot \begin{bmatrix}-\sin\theta_M\\ \cos\theta_M\end{bmatrix}\mathrm{d}\sigma\nonumber\\
&= -\frac{\mu s}{2} e^{\mathrm{i}\theta_M/2}\begin{bmatrix}\mathrm{i}\\ -1 \end{bmatrix} \int_0^h r^{-\frac{1}{2}} e^{-s\sqrt{r}\hat{\mu}(\theta_M)}\mathrm{d}r\nonumber\\
&= -\mu s e^{\mathrm{i}\theta_M/2}\begin{bmatrix}\mathrm{i}\\ -1 \end{bmatrix} \int_0^{\sqrt{h}} e^{-s\,t\,\hat{\mu}(\theta_M)}\mathrm{d}t\nonumber\\
&=\mu\big(e^{-s\,\sqrt{h}\,\hat{\mu}(\theta_M)}-1\big)\begin{bmatrix}\mathrm{i}\\-1\end{bmatrix}.
\end{align}
By similar arguments, we can derive
\begin{align}
\int_{\Gamma^-_h}  \mathcal{T}_{\bnu_m}\bu_0  \,\mathrm {d}\sigma&=
\mu\big(e^{-s\sqrt{h}\,\hat{\mu}(\theta_m)}-1\big)\begin{bmatrix}-\mathrm{i}\\1\end{bmatrix}.\label{eq:esti2_2}
\end{align}
From Lemma~\ref{10-lem:u0 int}, we obtain
{\small\begin{align}
&\int_{\Gamma^+_{h} }  \bu_0 \, {\rm d} \sigma =2 s^{-2}\left( \hat{\mu}^{-2}(\theta_M )-   \hat{\mu}^{-2}(\theta_M ) e^{ -s\sqrt{h}\, \hat{\mu}(\theta_M ) }-  \hat{\mu}^{-1}(\theta_M ) s\sqrt{h}   e^{ -s\sqrt{h} \,\hat{\mu}(\theta_M ) }  \right  )\begin{bmatrix}1\\ \mathrm{i}\end{bmatrix}, \label{inte1} \\[10pt]
&\int_{\Gamma^-_{h}  }  \bu_0 \,  {\rm d} \sigma =2 s^{-2} \left( \hat{\mu}^{-2}(\theta_m )-   \hat{\mu}^{-2}(\theta_m ) e^{ -s\sqrt{h}\,\hat{\mu}(\theta_m )}-  \hat{\mu}^{-1}(\theta_m ) s\sqrt{h}   e^{ -s\sqrt{h}\,\hat{\mu}(\theta_m ) }  \right  )\begin{bmatrix}1\\ \mathrm{i}\end{bmatrix}.\label{inte1'} 	
\end{align}}
Substituting \eqref{eq:esti2_1}--\eqref{inte1'} into \eqref{eq:iden2}, the following integral identity holds,
\begin{align}\label{eq:iden1}
&\Re\bff_+({\bf 0})\cdot \mu\Big(e^{-s\sqrt{h}\,\hat{\mu}(\theta_M)}-1\Big)\begin{bmatrix}\mathrm{i}\\-1\end{bmatrix}+\Re\bff_-({\bf 0})\cdot \mu\Big(e^{-s\sqrt{h}\,\hat{\mu}(\theta_m)}-1\Big)\begin{bmatrix}-\mathrm{i}\\+1\end{bmatrix}\nonumber\\
&-\Re\bg_+({\bf 0})\cdot\begin{bmatrix}1\\\mathrm{i}\end{bmatrix}2s^{-2}\hat{\mu}^{-2}(\theta_M)-\Re\bg_-({\bf 0})\cdot\begin{bmatrix}1\\\mathrm{i}\end{bmatrix}2s^{-2}\,\hat{\mu}^{-2}(\theta_m)=\sum\limits_{j=1}^{9}R_j,
\end{align}
where
\begin{align*}
&R_1=-2s^{-2}\Re\bg_+({\bf 0})\cdot\begin{bmatrix} 1\\ \mathrm{ i}\end{bmatrix}\Big(\hat{\mu}^{-1}(\theta_M)\,s\sqrt{h}\,e^{-s\sqrt{h}\,\hat{\mu}(\theta_M)}+\hat{\mu}^{-2}(\theta_M)\,e^{-s\sqrt{h}\,\hat{\mu}(\theta_M)}\Big),\\
&R_2=-2s^{-2}\Re\bg_-({\bf 0})\cdot\begin{bmatrix} 1\\ \mathrm{ i}\end{bmatrix}\big(\hat{\mu}^{-1}(\theta_m)\,s\sqrt{h}\,e^{-s\sqrt{h}\,\hat{\mu}(\theta_m)}+\hat{\mu}^{-2}(\theta_m)\,e^{-s\sqrt{h}\,\hat{\mu}(\theta_m)}\big),\\
&R_3=\int_{\Lambda_h}\mathcal{T}_{\bnu}\big(\Re\bv-\Re\bw\big)\cdot \bu_0-\mathcal{T}_{\bnu}\bu_0\cdot \big(\Re\bv-\Re\bw\big)\,\mathrm{d}\sigma,\\
&R_4=-\int_{\Gamma^+_h}\Re\delta\bff_+\cdot \mathcal{T}_{\bnu_M}\bu_0\,\mathrm{d}\sigma,\quad\,\, R_8=-\int_{\Gamma^+_h}\big(\nabla\Re\bff_+({\bf 0})\cdot \bx\big)\cdot\mathcal{T}_{\bnu_M}\bu_0\,\mathrm{d}\sigma,\\
&R_6=\int_{\Gamma^+_h}\Re\delta\bg_+\cdot \bu_0\,\mathrm{d}\sigma,\qquad\qquad
R_9=-\int_{\Gamma^-_h}\big(\nabla\Re\bff_-({\bf 0})\cdot \bx\big)\cdot \mathcal{T}_{\bnu_m}\bu_0\,\mathrm{d}\sigma,\\
&R_5=-\int_{\Gamma^-_h}\Re\delta\bff_-\cdot \mathcal{T}_{\bnu_m}\bu_0\,\mathrm{d}\sigma,\quad\,\,\, R_7=\int_{\Gamma^-_h}\Re\delta\bg_-\cdot \bu_0\,\mathrm{d}\sigma.
\end{align*}
From the expression of $\hat{\mu}(\cdot)$ given by Lemma~\ref{10-lem:u0 int}, it is direct to obtain that
$\hat{\mu}^{-2}(\theta_M)$, $\hat{\mu}^{-1}(\theta_M)$, $\hat{\mu}^{-2}(\theta_m)$ and $\hat{\mu}^{-1}(\theta_m)$ are bounded.
For sufficient large $s$, we have

\begin{align}\label{eq:R1}
\big|R_1\big|=\Oh(s^{-1}e^{-c_1 s})\quad \mbox{and}\quad \big|R_2\big|&=\Oh(s^{-1}e^{-c_2 s}),
\end{align}
where $c_1$ and $c_2$ are  positive constants not depending on $s$.
Considering the estimates of $\delta\bff_+$ in \eqref{eq:hollder5}, the expression of $\mathcal{T}_{\bnu_M}\bu_0$ in \eqref{eq:esti2_1} and Lemma~\ref{10-lem:u0 int1},
we get
\begin{align}\label{eq:R1'}
\big|R_4\big|\leq\,c_4\,\int_0^h r^{\alpha_++1}e^{-s\sqrt{r}\cos\frac{\theta_M}{2}}\,\mathrm{d}r= \Oh(s^{-2\alpha_+-2} ).
\end{align}
Similarly, we get the following estimates
\begin{align}
&\big|R_5\big|= \Oh(s^{-2\alpha_--2} ),\quad \big|R_{6}\big|= \Oh(\tau^{-2\beta_+-2} ),\quad 
\big|R_{7}\big|= \Oh(\tau^{-2\beta_--2} ),\nonumber\\ 
&\big|R_8\big|=\Oh(\tau^{-2} ), \quad \qquad \big|R_9\big|=\Oh(\tau^{-2} ).\quad \label{eq:R2}
\end{align}
For the term $R_3$, by virtue of  Cauchy-Schwarz inequality, trace theorem and Lemma~\ref{10-lem:23}, we obtain
\begin{align}
\big|R_3  \big|&\leq \|\bu_0\|_{L^2(\Lambda_h)^2}\|\mathcal{T}_{\bnu}(\Re\bv-\Re\bw)\|_{L^2(\Lambda_h)^2}+\|\Re\bv-\Re\bw\|_{L^2(\Lambda_h)^2}\|\mathcal{T}_{\bnu} \bu_0\|_{L^2(\Lambda_h)^2}\nonumber\\
&\leq \Big( \|\bu_0\|_{L^2(\Lambda_h)^2}+\|\mathcal{T}_{\bnu}\bu_0\|_{L^2(\Lambda_h)^2}\Big)\big\|\bv-\bw\big\|_{H^1(\mathcal{C}_h)^2}\nonumber\\
&\leq c_3\big\|\bu_0\big\|_{H^1(\mathcal{C}_h)^2}\leq c_3 e^{-s \sqrt{h}\delta_{\mathcal{K}}},\label{eq:R4}
\end{align}
where $c_3$ is positive and $\delta_{\mathcal{K}}$ is defined in \eqref{eq:lame6}.
Together \eqref{eq:R1}--\eqref{eq:R4} with the identity~\eqref{eq:iden1}, as $s\rightarrow +\infty$, one clearly implies that
\begin{equation*}
\begin{bmatrix}+\mathrm{i}\\-1\end{bmatrix}\cdot\Big(\Re\bff_+({\bf 0})-\Re\bff_-({\bf 0})\Big)=0.
\end{equation*}
Since $\Re\bff_+({\bf 0})$ and $\Re\bff_-({\bf 0})$ belong to $\mathbb{R}$, we derive that
\beq\label{eq:R6}
\Re\bff_+({\bf 0})=\Re\bff_-({\bf 0}).
\eeq
Hence, it implies that \eqref{eq:local2} holds.

Moreover, since $\nabla\bff_+({\bf 0})=\nabla\bff_({\bf 0})={\bf 0}$, by substituting \eqref{eq:R6} into \eqref{eq:iden1}, and subsequently multiplying the new equation by $s^2$, we arrive at
{\small \begin{align*}
&\mu\, s^2\Re\bff_+({\bf 0})\cdot \begin{bmatrix}\mathrm{i}\\-1\end{bmatrix}\bigg(e^{-s\sqrt{h}\,\hat{\mu}(\theta_M)}-e^{-s\sqrt{h}\,\hat{\mu}(\theta_m)}\bigg)-2\begin{bmatrix}1\\\mathrm{i}\end{bmatrix}\cdot\bigg(\frac{\Re\bg_+({\bf 0})}{\hat{\mu}^2(\theta_M)}+\frac{\Re\bg_-({\bf 0})}{\hat{\mu}^2(\theta_m)}\bigg)=\sum\limits_{j=1}^{7}s^2\,R_j,
\end{align*}}
 Let $s$ tend to $+\infty$, we have
\begin{align*}
\begin{bmatrix} 1\\ \mathrm{i}\end{bmatrix} \cdot\bigg(\Re\bg_+({\bf 0})\hat{\mu}^{-2}(\theta_M)+\Re\bg_-({\bf 0}) \hat{\mu}^{-2}(\theta_m)\bigg) =0.
\end{align*}
Noting that $\hat{\mu}^{-2}(\theta_M)\neq \hat{\mu}^{-2}(\theta_m)$ and $\frac{\hat{\mu}^2(\theta_M)}{\hat{\mu}^2(\theta_m)}=e^{\mathrm{i}(\theta_M-\theta_m)}$, then let $\Re\bg_+({\bf 0}):=\begin{bmatrix} a_{11}\\a_{21}\end{bmatrix}$ and $\Re\bg_-({\bf 0}):=\begin{bmatrix} a_{12}\\a_{22}\end{bmatrix}$, where $a_{ij}\in \mathbb{R}$ for $i,j=1,2$. So the above equation can be rewritten as follows
\allowdisplaybreaks
\beq\nonumber
\begin{cases}
a_{11}+a_{12}\cos(\theta_M-\theta_m)-a_{22}\sin(\theta_M-\theta_m) =0,\\
a_{21}+a_{12}\sin(\theta_M-\theta_m)+a_{22}\cos(\theta_M-\theta_m) =0,
\end{cases}
\eeq
namely,
$$
\Re\bg_+({\bf 0})=W \Re\bg_-({\bf 0}).
$$
Therefore, we obtain the identity \eqnref{eq:local2'}.
Here, 
$$
W=\begin{bmatrix}-\cos(\theta_M-\theta_m),&-\sin(\theta_M-\theta_m)\\
    -\sin(\theta_M-\theta_m), &\cos(\theta_M-\theta_m)
    \end{bmatrix}\quad\mbox{and}\quad \det(W)\neq 0.
$$
The proof is complete.
\end{proof}

\subsection{Local results for 3D case}\label{subse:case2}
As discussed in Remark~4.2 in \cite{DiaoLiuSun2021}, the regularity result on the underlying elastic displacement around a general polyhedral corner in $\mathbb{R}^3$ is challenging. So in this subsection, we shall restrict ourselves to the 3D edge corner in the usual sense as the one for the 2D case. 
We introduce a dimension reduction operator $\mathcal{P}$ to study the relevant continuities about jump functions $\bff$ and $\bg$ at the 3D edge corner. In what follows, we suppose that the dislocation $\mathcal{S}$ is a Lipschitz surface possessing at least one 3D edge corner point $\bx_c=(\bx'_c,x^3_c)^\top\in\mathcal{S}\subset\mathbb{R}^3$.

The next definition shall state a dimension reduction operator, which is beneficial to derive a crucial auxiliary proposition similar to Proposition~\ref {prop:trans2} at a 3D edge corner.

\begin{defn}\label{def:operator}
Let $\mathcal{C}'_{\bx'_c,h}\subset \mathbb{R}^2$ be defined in \eqref{eq:ball1} with the vertex $\bx'_c$ and $M>0$. Consider a given function $\phi\in\mathcal{C}'_{\bx'_c,h}\times (-M,M)$ and pick up any point $x^3_c\in (-M,M)$. For a sufficiently small constant $L$ satisfying $(x^3_c-L,x^3_c+L)\subset (-M,M)$, we assume that $\phi\in C^\infty_0(x^3_c-L,x^3_c+L)$ is a nonnegative function and $\phi \not\equiv 0$. Write $\bx=(\bx',x_3)^\top\in \mathbb{R}^3$. The dimension reduction operator $\mathcal{P}$ is defined as follows
\beq\nonumber
\mathcal{P}(\bh)(\bx')=\int_{x^3_c-L}^{x_c^3+L}\phi(x_3)\bh(\bx',x_3)\mathrm{d}x_3.
\eeq
\end{defn}
Before deriving the main results of this subsection, we review some important properties of such an operator.
\begin{lem}\cite[Lemma~3.1]{DiaoLiuSun2021}
Let $\bh\in H^m(\mathcal{C}'_{\bx'_c,h}\times (-M,M))^3$, $m=1,2$. Then
\beq\nonumber
\mathcal{P}(\bh)(\bx')\in H^m(\mathcal{C}'_{\bx'_c,h}))^3.
\eeq
Similarly, if $\bh\in C^\delta\big(\overline{\mathcal{C}'_{\bx'_c,h}}\times [-M,M]\big)^3$ with $\delta\in (0,1)$, then
\beq\nonumber
\mathcal{P}(\bh)(\bx')\in C^\delta\big(\overline{\mathcal{C}'_{\bx'_c,h}}\big)^3.
\eeq
\end{lem}

Noting that the three-dimensional isotropic elastic operator $\mathcal{L}$ defined in \eqref{eq:tensor1}--\eqref{eq:op1} can be rewritten as
{\small\begin{align*}
\mathcal{L}&=\begin{bmatrix}
\lambda\Delta +(\lambda+\mu)\partial_1^2&(\lambda+\mu)\partial_1\partial_2&(\lambda+\mu)\partial_1\partial_3\\
(\lambda+\mu)\partial_1\partial_2&\lambda\Delta +(\lambda+\mu)\partial_2^2&(\lambda+\mu)\partial_2\partial_3\\
(\lambda+\mu)\partial_1\partial_3&(\lambda+\mu)\partial_2\partial_3&\lambda\Delta +(\lambda+\mu)\partial_3^2
\end{bmatrix}\\[5mm]
&=\widetilde{\mathcal{L}}+
\begin{bmatrix}
\lambda\partial^2_3&0&(\lambda+\mu)\partial_1\partial_3\\
0&\lambda\partial_3^2&(\lambda+\mu)\partial_2\partial_3\\
(\lambda+\mu)\partial_1\partial_3&(\lambda+\mu)\partial_2\partial_3&\lambda\partial_3^2+(\lambda+\mu)\partial_3^2
\end{bmatrix},
\end{align*}}
where
\beq\label{eq:L}
\widetilde{\mathcal{L}}=\begin{bmatrix}
\lambda\Delta'+(\lambda+\mu)\partial_1^2&(\lambda+\mu)\partial_1\partial_2&0\\
(\lambda+\mu)\partial_1\partial_2&\lambda\Delta'+(\lambda+\mu)\partial_2^2&0\\
0&0&\lambda\Delta'
\end{bmatrix}=\begin{bmatrix}
\mathcal{L}_{\mathcal{P}}&0\\
0&\lambda\Delta'
\end{bmatrix}
\eeq
with $\Delta'=\partial_1^2+\partial_2^2$ being the Laplace operator with respect to the $\bx'$-variables. Here, the operator $\mathcal{L}_{\mathcal{P}}$ is the two-dimensional isotropic elastic operator with respect to the $\bx'$-variables.

For any fixed $x_c^3\in (-M, M)$ and sufficient small $L>0$ such that $(x_c^3-L,x_c^3+L)\subset (-M,M)$. At this moment, we have $\mathcal{L}=\widetilde{\mathcal{L}}$. Since $\mathcal L$ is invariant under rigid motion, in what follows we assume that $\bx'_c=\mathbf 0$ in $\mathbb R^2.$ Hence let $\mathcal{C}'_{\bx'_c,h}$ and $\Gamma^\pm_{\bx'_c,h}$ be defined in \eqref{eq:ball1} with $\bx'_c$ coinciding with the origin in 2D case. By some tedious calculations, we have the following lemma.
\begin{lem}\label{lem:R3_1}
Under the setup about $\mathcal{C}'_{\bx'_c,h}$ and a 3D edge corner as above, suppose that $\bff_j\in C^{1,\alpha_j}\big(\Gamma'^j_{\bx'_c,h}\times (-M,M)\big)^3\bigcap H^{\frac{1}{2}}\big(\Gamma'^j_{\bx'_c,h}\times (-M,M)\big)^3$  and $\bg_j\in C^{\beta_j}\big(\Gamma'^j_{\bx'_c,h}\times (-M,M)\big)^3\bigcap H^{-\frac{1}{2}}\big(\Gamma'^j_{\bx'_c,h}\times (-M,M)\big)^3$ do not depend on $x_3$ for  $j=+,-$, where $\alpha_+$, $\alpha_-$, $\beta_+$, $\beta_- \in (0,1)$. Denote $\bv=(\bv^{(1,2)}, \,\, v_3)^\top$, $\bw=(\bw^{(1,2)},\,\, w_3)^\top$, $\bff_\pm=(\bff^{(1,2)}_\pm,\,\, f^3_\pm)^\top$ and $\bg_\pm=(\bg^{(1,2)}_\pm,\,\, g^3_\pm)^\top$. Then the  problem for $(\bv,\bw)\in H^1\big(\mathcal{C}'_{\bx'_c,h}\times (-M,M)\big)^3\times  H^1\big(\mathcal{C}'_{\bx'_c,h}\times (-M,M)\big)^3:$
\allowdisplaybreaks
\begin{equation*}
\begin{cases}
\mathcal{L}\,\bv=\mathbf{0},\hspace{15pt} \mathcal{L}\,\bw=\mathbf{0}  &\mbox{in}\quad \mathcal{C}'_{\bx'_c,h}\times (-M,M),\\[2mm]
\bv-\bw=\bff_+,\,\mathcal{T}_{\bnu}\,\bv-\mathcal{T}_{\bnu}\,\bw=\bg_+ &\mbox{on}\quad\Gamma'^+_{\bx'_c,h}\times (-M,M),\\[2mm]
\bv-\bw=\bff_-,\,\mathcal{T}_{\bnu}\,\bv-\mathcal{T}_{\bnu}\,\bw=\bg_- &\mbox{on}\quad\Gamma'^-_{\bx'_c,h}\times (-M,M)
\end{cases}
\end{equation*}
can be reduced  to be
\begin{equation}\label{eq:trans3D'}
\begin{cases}
\widetilde{\mathcal{L}}\,\mathcal{P}(\bv)(\bx')=\mathcal{G}_1(\bx'),\hspace{8pt} \widetilde{\mathcal{L}}\,\mathcal{P}(\bw)(\bx')=\mathcal{G}_2(\bx'),  &\bx'\in \mathcal{C}'_{\bx'_c,h},\\[2mm]
\mathcal{P}(\bv)(\bx')=\mathcal{P}(\bw)(\bx')+\mathcal{P}(\bff_+)(\bx'), &\bx'\in \Gamma'^+_{\bx'_c,h},\\[2mm]
\mathcal{P}(\bv)(\bx')=\mathcal{P}(\bw)(\bx')+\mathcal{P}(\bff_-)(\bx'), &\bx'\in \Gamma'^-_{\bx'_c,h},\\[2mm]
\mathcal{R}^+_1
=\mathcal{R}^+_2+\mathcal{P}(\bg_+)(\bx'),&\bx'\in \Gamma'^+_{\bx'_c,h},\\[2mm]
\mathcal{R}^-_1
=\mathcal{R}^-_2+\mathcal{P}(\bg_-)(\bx'),&\bx'\in \Gamma'^-_{\bx'_c,h},
\end{cases}
\end{equation}
where
{\small\begin{align*}
&\mathcal{G}_1=-\int^{x_c^3+L}_{x_c^3-L}\phi''(x_3)\begin{bmatrix} \lambda \bv^{(1,2)}(\bx)\\ (2\lambda+\mu)v_3(\bx)  \end{bmatrix}\mathrm{d}x_3+(\lambda+\mu)\int_{x_c^3-L}^{x_c^3+L}\phi'(x_3)\begin{bmatrix} \nabla v_3(\bx)\\ \partial_1 v_1(\bx)+\partial_2 v_2(\bx)  \end{bmatrix}\mathrm{d}x_3,\\[3mm]
&\mathcal{G}_2=-\int^{x_c^3+L}_{x_c^3-L}\phi''(x_3)\begin{bmatrix} \lambda \bw^{(1,2)}(\bx)\\ (2\lambda+\mu)w_3(\bx)  \end{bmatrix}\mathrm{d}x_3+(\lambda+\mu)\int_{x_c^3-L}^{x_c^3+L}\phi'(x_3)\begin{bmatrix} \nabla w_3(\bx)\\ \partial_1 w_1(\bx)+\partial_2 w_2(\bx)  \end{bmatrix}\mathrm{d}x_3,\\[3mm]
&\mathcal{R}^{\mathrm{+}}_1=\begin{bmatrix}
\mathcal{T}_{\bnu_M}\mathcal{P}(\bv^{(1,2)})+\lambda\mathcal{P}(\partial_3 v_3)\bnu_M\\[1mm]
\mu\partial_{\bnu_M}\mathcal{P}(v_3)+\mu\begin{bmatrix}\mathcal{P}(\partial_3 v_1)\\ \mathcal{P}(\partial_3 v_2)
\end{bmatrix}\cdot \bnu_M
\end{bmatrix},\,\mathcal{R}^{\mathrm{+}}_2=\begin{bmatrix}
\mathcal{T}_{\bnu_M}\mathcal{P}(\bw^{(1,2)})+\lambda\mathcal{P}(\partial_3 w_3)\bnu_M\\[1mm]
\mu\partial_{\bnu_M}\mathcal{P}(w_3)+\mu\begin{bmatrix}\mathcal{P}(\partial_3 w_1)\\ \mathcal{P}(\partial_3w_2)
\end{bmatrix}\cdot \bnu_M
\end{bmatrix} \,\mbox{on}\, \Gamma'^{\mathrm{+}}_{\bx'_c,h},
\\[3mm]
&\mathcal{R}^{\mathrm{-}}_1=\begin{bmatrix}
\mathcal{T}_{\bnu_m}\mathcal{P}(\bv^{(1,2)})+\lambda\mathcal{P}(\partial_3 v_3)\bnu_m\\[1mm]
\mu\partial_{\bnu_m}\mathcal{P}(v_3)+\mu\begin{bmatrix}\mathcal{P}(\partial_3 v_1)\\ \mathcal{P}(\partial_3 v_2)
\end{bmatrix}\cdot \bnu_m
\end{bmatrix},\,\mathcal{R}^{\mathrm{-}}_2=\begin{bmatrix}
\mathcal{T}_{\bnu_m}\mathcal{P}(\bw^{(1,2)})+\lambda\mathcal{P}(\partial_3 w_3)\bnu_m\\[1mm]
\mu\partial_{\bnu_m}\mathcal{P}(w_3)+\mu\begin{bmatrix}\mathcal{P}(\partial_3 w_1)\\ \mathcal{P}(\partial_3w_2)
\end{bmatrix}\cdot \bnu_m
\end{bmatrix} \,\mbox{on}\, \Gamma'^{\mathrm{-}}_{\bx'_c,h}.
\end{align*}} 
Here, $\bnu_M$ and $\bnu_m$ denote the exterior unit normal vector to $\Gamma'^+_{\bx'_c,h}$ and $\Gamma'^+_{\bx'_c,h}$, respectively, $\mathcal{T}_{\bnu}$ is the two-dimensional boundary traction operator.
\end{lem}

By applying the decomposition of $\widetilde{\mathcal{L}}$ given in \eqref{eq:L}, it is direct to obtain the following results. Here, we omit the proof.

\begin{lem}
Under the same setup in Lemma~\ref{lem:R3_1}, the PDE system \eqref{eq:trans3D'} is equivalent to the following two PDE systems
\begin{equation}\label{eq:trans3D_1}
\begin{cases}
\mathcal{L}_{\mathcal{P}}\,\mathcal{P}(\bv^{(1,2)})(\bx')=\mathcal{G}^{(1,2)}_1(\bx')\quad &\mbox{in}\quad \mathcal{C}'_{\bx'_c,h},\\[2mm]
\mathcal{L}_{\mathcal{P}}\,\mathcal{P}(\bw^{(1,2)})(\bx')=\mathcal{G}^{(1,2)}_2(\bx') \quad &\mbox{in}\quad \mathcal{C}'_{\bx'_c,h},\\[2mm]
\mathcal{P}(\bv^{(1,2)})(\bx')=\mathcal{P}(\bw^{(1,2)})(\bx')+\mathcal{P}(\bff_+^{(1,2)})(\bx') \quad&\mbox{on}\quad \Gamma'^+_{\bx'_c,h},\\[2mm]
\mathcal{P}(\bv^{(1,2)})(\bx')=\mathcal{P}(\bw^{(1,2)})(\bx')+\mathcal{P}(\bff_-^{(1,2)})(\bx') \quad&\mbox{on}\quad \Gamma'^-_{\bx'_c,h},\\[2mm]
\mathcal{R}^{(1,2)}_1
=\mathcal{R}^{(1,2)}_2+\mathcal{P}(\bg_+^{(1,2)})(\bx')\quad&\mbox{on}\quad \Gamma'^+_{\bx'_c,h},\\[2mm]
\mathcal{R}^{(1,2)}_1
=\mathcal{R}^{(1,2)}_2+\mathcal{P}(\bg_-^{(1,2)})(\bx')\quad&\mbox{on}\quad \Gamma'^-_{\bx'_c,h}
\end{cases}
\end{equation}
and
\begin{equation}\label{eq:trans3D_2}
\begin{cases}
\lambda\Delta'\,\mathcal{P}(v_3)(\bx')=\mathcal{G}^{(3)}_1(\bx')&\mbox{in}\quad  \mathcal{C}'_{\bx'_c,h},\\[2mm] \lambda\Delta'\,\mathcal{P}(w_3)(\bx')=\mathcal{G}^{(3)}_2(\bx')  &\mbox{in}\quad  \mathcal{C}'_{\bx'_c,h},\\[2mm]
\mathcal{P}(v_3)(\bx')=\mathcal{P}(w_3)(\bx')+\mathcal{P}(f^3_+)(\bx')&\mbox{on}\quad  \Gamma'^+_{\bx'_c,h},\\[2mm]
\mathcal{P}(v_3)(\bx')=\mathcal{P}(w_3)(\bx')+\mathcal{P}(f^3_-)(\bx')&\mbox{on}\quad  \Gamma'^-_{\bx'_c,h},\\[2mm]
\partial_{\bnu_M}\mathcal{P}(v_3)(\bx')
=\partial_{\bnu_M}\mathcal{P}(w_3)(\bx')+\frac{1}{\mu}\mathcal{P}(g^3_+)(\bx')\qquad\quad&\mbox{on}\quad  \Gamma'^+_{\bx'_c,h},\\[2mm]
\partial_{\bnu_m}\mathcal{P}(v_3)(\bx')
=\partial_{\bnu_m}\mathcal{P}(w_3)(\bx')+\frac{1}{\mu}\mathcal{P}(g^3_-)(\bx')\qquad\quad&\mbox{on}\quad  \Gamma'^-_{\bx'_c,h},
\end{cases}
\end{equation}where
\allowdisplaybreaks

\begin{align*}
&\mathcal{G}^{(1,2)}_1=-\lambda \int^{x_c^3+L}_{x_c^3-L}\phi''(x_3) \bv^{(1,2)}(\bx)\mathrm{d}x_3+(\lambda+\mu)\int_{x_c^3-L}^{x_c^3+L}\phi'(x_3)\nabla v_3(\bx)\mathrm{d}x_3,\\[1mm]
&\mathcal{G}^{(1,2)}_2=- \lambda \int^{x_c^3+L}_{x_c^3-L}\phi''(x_3) \bw^{(1,2)}(\bx)\mathrm{d}x_3+(\lambda+\mu)\int_{x_c^3-L}^{x_c^3+L}\phi'(x_3)\nabla w_3(\bx)\mathrm{d}x_3,\\[1mm]
&\mathcal{G}^{(3)}_1=- (2\lambda+\mu)\int^{x_c^3+L}_{x_c^3-L}\phi''(x_3) v_3(\bx)\mathrm{d}x_3+(\lambda+\mu)\int_{x_c^3-L}^{x_c^3+L}\phi'(x_3)(\partial_1 v_1+\partial_2 v_2) \mathrm{d}x_3,\\[1mm]
&\mathcal{G}^{(3)}_2=- (2\lambda+\mu)\int^{x_c^3+L}_{x_c^3-L}\phi''(x_3) w_3(\bx)\mathrm{d}x_3+(\lambda+\mu)\int_{x_c^3-L}^{x_c^3+L}\phi'(x_3)(\partial_1 w_1+\partial_2 w_2) \mathrm{d}x_3,\\[1mm]
&\mathcal{R}^{+(1,2)}_1=\mathcal{T}_{\bnu_M}\mathcal{P}(\bv^{(1,2)})+\lambda\mathcal{P}(\partial_3 v_3)\,\bnu_M,\,\, \mathcal{R}^{+(1,2)}_2=\mathcal{T}_{\bnu_M}\mathcal{P}(\bw^{(1,2)})+\lambda\mathcal{P}(\partial_3 w_3) \hspace{10pt}\mbox{on}\,\,  \Gamma'^+_{\bx'_c,h}, \\[1mm]
&\mathcal{R}^{-(1,2)}_1=\mathcal{T}_{\bnu_m}\mathcal{P}(\bv^{(1,2)})+\lambda\mathcal{P}(\partial_3 v_3)\,\bnu_m,\,\, \, \mathcal{R}^{-(1,2)}_2=\mathcal{T}_{\bnu_m}\mathcal{P}(\bw^{(1,2)})+\lambda\mathcal{P}(\partial_3 w_3) \hspace{12pt}\mbox{on}\,\,  \Gamma'^-_{\bx'_c,h}.
\end{align*}
\end{lem}

Next, we shall recall another CGO solution $u_0$, which is introduced in \cite{Blasten2018} and exhibits a similar form and properties as those detailed in Lemma~\ref{lem:cgo1}--Lemma~\ref{10-lem:u0 int} (cf. \cite{DiaoCaoLiu2021}). Indeed, this CGO solution is employed to investigate the continuity of $\mathcal{P}(f_j^{3})$ in the conventional sense and $\mathcal{P}(g_j^{3})$ in the rotational sense at a 3D edge corner, where $j=+,-$. To be more precise, \eqref{eq:vansh2} and \eqref{eq:vansh2'} provide further details. For the reader's convenience, we have listed the expression of $u_0$ and its associated properties here.

\begin{lem}\label{lem:cgo2} \cite[Lemma 2.2]{Blasten2018}
Let $\bx'=(x_1,x_2)^\top=r(\cos\theta,\sin\theta)^\top\in \mathbb{R}^2$ and $s\in \mathbb{R}_+$,
\beq\label{cgo2}
u_0(s\bx'):=\exp (-\sqrt{sr}\hat{\mu}(\theta)),
\eeq
where $\hat{\mu}(\cdot)$ is defined in Lemma~\ref{10-lem:u0 int}. Then $s\longmapsto u_0(s\bx')$ decays exponentially in $\mathbb{R}_+$ and
\begin{equation*}
\Delta' u_0=0\quad\mbox{in}\quad \mathbb{R}^2\backslash\mathbb{R}^2_{0,-}
\end{equation*}
where $\mathbb{R}^2_{0,-}:=\{\bx'\in\mathbb{R}^2|\bx'=(x_1,x_2);x_1\leq 0,x_2=0\}$. Moreover,
\begin{equation*}
\int_\mathcal{K_\mathcal{P}} u_0(s\bx'){\rm d}\bx'=6{\rm i}(e^{-2\theta_M{\rm i}}-e^{-2\theta_m \mathrm i})s^{-2}
\end{equation*}
and for $\alpha,\,s>0$ and $h>0$
\allowdisplaybreaks
\begin{align*}
\int_{\mathcal{K}_{\mathcal{P}}}|u_0(s\bx')||\bx'|^{\alpha}{\rm d}\bx'&\leq\frac{2(\theta_M-\theta_m)\Gamma(2\alpha+4)}{\delta_{\mathcal{K}_\mathcal{P}}^{2\alpha+4}}s^{-\alpha-2},\nonumber\\
 \int_{\mathcal{K}_{\mathcal{P}}\backslash B_h}|u_0(s\bx')|{\rm d}\bx'&\leq\frac{6(\theta_M-\theta_m)}{\delta_{\mathcal{K}_\mathcal{P}}^4}s^{-2}e^{-\frac{\sqrt{sh}}{2}\delta_{\mathcal{K}_\mathcal{P}}} , 
\end{align*}
where $\mathcal{K}_\mathcal{P}$ is defined like $\mathcal{K}$ in Section~\ref{sec:1} and $\delta_{\mathcal{K}_\mathcal{P}}=\min\limits_{\theta_m<\theta<\theta_M} \cos \frac{\theta}{2}$ is a positive constant.
\end{lem}

\begin{lem}\cite[Lemma 2.4]{DiaoCaoLiu2021}
For any $\alpha>0$, if $\omega(\theta)>0$, then we have
$$
\lim\limits_{s\rightarrow +\infty}\int^h_0 r^{\alpha}e^{-\sqrt{sr}\omega(\theta)}\mathrm{d}\mathrm{r}=\Oh(s^{-\alpha-1}).
$$
\end{lem}

\begin{lem}\label{10-lem:23'}\cite[Lemma~2.3]{DiaoCaoLiu2021}
	Let $\mathcal{C}'_{\bx'_c,h}$ be defined as in \eqref{eq:ball1} and $u_0$ be given in \eqref{cgo2}. Then $u_0 \in H^1(\mathcal{C}'_{\bx'_c,h})^2$ and $\Delta' u_0 = 0$ in $\mathcal{C}'_{\bx'_c,h}$. Furthermore, it holds that
	\begin{equation*}
	\big\|u_0\big\|_{L^2(\mathcal{C}'_{\bx'_c,h})^2 } \leq 	 \frac{\sqrt{\theta_M-\theta_m} e^{- 2\sqrt{s\Theta } \,\delta_{\mathcal{K}_\mathcal{P}} }h^2}{2}
	\end{equation*}
	and
	\begin{equation*}
		\Big  \||\bx'|^\alpha  u_0 \Big \|^2_{L^{2}(\mathcal{C}'_{\bx'_c,h} )^2  }\leq s^{-2(\alpha+1 )} \frac{2(\theta_M-\theta_m)\Gamma(4\alpha+4)    }{(4\delta_{\mathcal{K}_\mathcal{P}})^{2\alpha+2  } } \,
	\end{equation*}
where $ \Theta  \in [0,h ]$ and $\delta_{\mathcal{K}_\mathcal{P}}$ is given in Lemma~\ref{lem:cgo2}.
\end{lem}
We continue to derive a key proposition to establish one main geometric result, which is a three-dimensional result similar to Proposition~\ref{prop:trans2}.
\begin{prop}\label{prop:trans4}
Consider the same setup in Lemma~\ref{lem:R3_1} with $\bx_c$ coinciding with the origin.  Assume that $\bff_j\in C^{1,\alpha_j}\big(\Gamma'^j_h\times (-M,M)\big)^3\bigcap H^{\frac{1}{2}}\big(\Gamma'^j_h\times (-M,M)\big)^3$  and $\bg_j\in C^{\beta_j}\big(\Gamma'^j_h\times (-M,M)\big)^3\bigcap H^{-\frac{1}{2}}\big(\Gamma'^j_h\times (-M,M)\big)^3$ are independent of $x_3$, where  $\alpha_j,\beta_j$ are in $(0,1)$ and $j=+,-$. Let $\bv\in H^1\left(\mathcal{C}'_h\times (-M,M)\right)^3$ and $\bw\in H^1\left(\mathcal{C}'_h\times (-M,M)\right)^3$ satisfy the PDE system~\eqref{eq:trans3D'}. Let $\bg_j({\bf 0})=(\bg^{(1,2)}_j({\bf 0})^\top,g^{3}_j({\bf 0}))^\top$ for $j=+,-$.  Then we get 
 \begin{equation}\label{eq:vansh2}
\bff_+({\bf 0})=\bff_-({\bf 0}).
 \end{equation}
 Furthermore, if $\nabla\bff_+({\bf 0})=\nabla\bff_-({\bf 0})={\bf 0}$ are satisfied, then it implies that 
  \begin{equation}\label{eq:vansh2'}
 \bg^{(1,2)}_+({\bf 0})=W\bg^{(1,2)}_-({\bf 0})\quad{and}\quad g_+^{3}({\bf 0})=g_-^{3}({\bf 0})=0,
  \end{equation}
 where $W=\begin{bmatrix}-\cos(\theta_M-\theta_m),&-\sin(\theta_M-\theta_m)\\
    -\sin(\theta_M-\theta_m), &+\cos(\theta_M-\theta_m)
    \end{bmatrix}$, $\theta_M$ and $\theta_m$ are the arguments corresponding to the boundaries $\Gamma'^+_h$ and $\Gamma'^-_h$ respectively.
\end{prop}
\begin{proof}
Similar to the proof of Proposition~\ref{prop:trans2}, we only consider the corresponding proofs for $(\Re\mathcal{P}(\bv),\Re\mathcal{P}(\bw))$. We follow similar arguments in the proof of Proposition~\ref{prop:trans2} with some necessary modifications. Since $\bff_j$ and $\bg_j$ ($j=+,-$) do not depend on $x_3$ and the Definition \ref{def:operator}, it is direct to obtain that $\mathcal{P}(\bff_j)=M_0\bff_j$ and $\mathcal{P}(\bg_j)=M_0\bg_j$, where $M_0=\int_{-L}^L\phi(x_3)\mathrm{d}x_3$. Next, we divide the proof into the following two parts.
\allowdisplaybreaks

\medskip
\noindent {\bf Part I.}~We first shall prove that
\begin{equation}\label{eq:vansh_2_1}
 \Re\mathcal{P}(\bff_+^{(1,2)})({\bf 0})=\Re\mathcal{P}(\bff_-^{(1,2)})({\bf 0}) \quad\mbox{and}\quad \Re\mathcal{P}(\bg_+^{(1,2)})({\bf 0})=W\Re\mathcal{P}(\bg^{(1,2)}_-)({\bf 0}),
 \end{equation} 
 where the second equation needs additional assumptions (i.e., $ \nabla\mathcal{P}(\bff^{(1,2)}_+)({\bf 0})=\nabla\mathcal{P}(\bff^{(1,2)}_-)({\bf 0})$ $={\bf 0}$). In this part, we consider the PDE system \eqref{eq:trans3D_1}. Note that $\bff_j\in C^{1,\alpha_j}\big(\Gamma'^j_h\times (-M,M)\big)^3$  and $\bg_j\in C^{\beta_j}\big(\Gamma'^j_h\times (-M,M)\big)^3$ for $j=+,-$. Since $\bff_j$ and $\bg_j$ do not depend on $x_3$, we derive the following expansions:
{\small\begin{alignat}{2}
 \mathcal{P}(\bff_j)(\bx')&=\mathcal{P}(\bff_j)({\bf0})+\nabla \mathcal{P}(\bff_j)({\bf 0})\cdot \bx'+ \delta\mathcal{P}(\bff_j)(\bx'),\quad &\big| \delta\mathcal{P}(\bff_j)(\bx') \big|&\leq A_j|\bx'|^{1+\alpha_j},\nonumber  \\
 \mathcal{P}(\bg_j)(\bx')&=\mathcal{P}(\bg)({\bf0})+ \delta\mathcal{P}(\bg)(\bx'),\quad &\big| \delta\mathcal{P}(\bg)(\bx')\big|&\leq B_j|\bx'|^{\beta_j}.\label{expansion2}
 \end{alignat}}
By a series of  similar derivations in Proposition~\ref{prop:trans2}, we deduce the following integral identity
 {\small\begin{align}
&\mu\,\Re\mathcal{P}(\bff_+^{(1,2)})({\bf 0})\cdot\begin{bmatrix} \mathrm{ i}\\-1 \end{bmatrix}\Big(e^{-s\sqrt{h}\,\hat{\mu}(\theta_M)}-1\Big)+\mu\,\Re\mathcal{P}(\bff_-^{(1,2)})({\bf 0})\cdot\begin{bmatrix} -\mathrm{ i}\\+1 \end{bmatrix}\Big(e^{-s\sqrt{h}\,\hat{\mu}(\theta_m)}-1\Big)\nonumber\\
&-2s^{-2}\Re\mathcal{P}(\bg_+^{(1,2)})({\bf 0})\cdot\begin{bmatrix} 1\\ \mathrm{ i}\end{bmatrix}\hat{\mu}^{-2}(\theta_M)-2s^{-2}\Re\mathcal{P}(\bg_-^{(1,2)})({\bf 0})\cdot\begin{bmatrix} 1\\ \mathrm{ i}\end{bmatrix}\hat{\mu}^{-2}(\theta_m)=\sum\limits_{j=1}^{10}Q_j,\label{eq:iden6}
\end{align}}
where $\hat{\mu}(\cdot)$ is given by Lemma~\ref{10-lem:u0 int}.
These $\Re\bff_j({\bf 0})$, $\Re\bg_j({\bf 0})$, $\Re\delta\bff_j$, $\Re\delta\bg_j$, $\Re\bv$ and $\Re\bw$ in $R_1$--$R_{9}$ given by \eqref{eq:iden1} are replaced by $\Re\mathcal{P}(\bff_j^{(1,2)})({\bf 0})$, $\Re\mathcal{P}(\bg_j^{(1,2)})({\bf 0})$, $\delta\Re\mathcal{P}(\bff_j^{(1,2)})$, $\delta\Re\mathcal{P}(\bg_j^{(1,2)})$, $\Re\mathcal{P}(\bv^{(1,2)})$ and $\Re\mathcal{P}(\bw^{(1,2)})$ for $j=+,-$.
In addition,
\begin{align*}
Q_{10}&=-\int _{\mathcal{C}'_h}\big(\Re\mathcal{G}^{(1,2)}_1(\bx')-\Re\mathcal{G}^{(1,2)}_2(\bx')\big)\cdot \bu_0\,\mathrm{d}{\bx'}.
\end{align*}
Similar to Proposition~\ref{prop:trans2}, we have the following estimates
\begin{align*}
&\big|Q_1\big|=\Oh(s^{-1}e^{-c'_1\,s}),\,\,\big|Q_2\big|=\Oh(s^{-1}e^{-c'_2\,s}), \,\big|Q_4\big|=\Oh(s^{-2\alpha_+-2}),\,\,
\big|Q_5\big|=\Oh(s^{-2\alpha_--2}),\\
&\big|Q_3\big|=\Oh(e^{-c_3's}),
\qquad\big|Q_6\big|=\Oh(s^{-2\beta_+-2}),
\,\,\,\big|Q_7\big|=\Oh(s^{-2\beta_--2} ),
\,\, \big|Q_8\big|=\big|Q_9\big|=\Oh(s^{-2} ),
\end{align*}
where these constants $c'_1,\,c'_2$ and $c'_3$ do not depend on $s$, $\alpha_+,\alpha_-,\,\beta_+,\,\beta_-\in (0,1)$. From the expressions of $\mathcal{G}_1$ and $\mathcal{G}_2$ defined in \eqref{eq:trans3D'}, denote $\Re\mathcal{G}^{(1,2)}_1-\Re\mathcal{G}^{(1,2)}_2:=\bh_1+\bh_2$, where 
$$
\bh_1=-\int_{-L}^{+L}\phi''(x_3)\begin{bmatrix}\lambda(\Re v_1-\Re w_1)\\ \lambda(\Re v_2-\Re w_2)\end{bmatrix}\mathrm{d}x_3,\,\, \bh_2=(\lambda+\mu)\int_{-L}^{+L}\phi'(x_3)\begin{bmatrix}\partial_1(\Re v_3-\Re w_3)\\ \partial_2(\Re v_3-\Re w_3)\end{bmatrix}\mathrm{d}x_3.
$$ 
By the regularities of $\bv$ and $\bw$ in $\mathcal{C}'_h\times (-M,M)$, we directly obtain that $\Re\mathcal{G}^{(1,2)}_1\in H^1(\mathcal{C}'_h)^2$ and $\Re\mathcal{G}^{(1,2)}_2\in L^2(\mathcal{C}'_h)^2$. By using the Cauchy-Schwarz inequality and the first equation in Lemma~\ref{10-lem:23}, we prove
\begin{align*}
&\bigg|\int_{\mathcal{C}'_h}(\Re\mathcal{G}^{(1,2)}_1-\Re\mathcal{G}^{(1,2)}_2)\cdot \bu_0\,\mathrm{d}x_3\,\bigg|=\bigg|\int_{\mathcal{C}'_h}(\bh_1+\bh_2)\cdot \bu_0\,\mathrm{d}x_3\,\bigg|\\[1mm]
&\,\,\leq \|\bh_1\|_{L^2(\mathcal{C}'_h)^2}\|\bu_0\|_{L^2(\mathcal{C}'_h)^2}+\|\bh_2\|_{L^2(\mathcal{C}'_h)^2}\|\bu_0\|_{L^2(\mathcal{C}'_h)^2}\\[1mm]
&\,\,\leq  C'e^{- c'_8s  },
\end{align*}
where $C'>0$ and $c'_8>0$  do not depend on $s$. From \eqref{eq:iden6}, we have
$$
\begin{bmatrix} -\mathrm{ i}\\1 \end{bmatrix}\cdot\Big(\Re\mathcal{P}(\bff_+^{(1,2)})({\bf 0})-\Re\mathcal{P}(\bff_-^{(1,2)})({\bf 0})\Big)=0 \quad \mbox{as}\quad s\rightarrow +\infty.
$$
Noticing that $\Re\mathcal{P}(\bff_+^{(1,2)})({\bf 0})$ and $\Re\mathcal{P}(\bff_-^{(1,2)})({\bf 0})$ are real, we can obtain \eqref{eq:vansh_2_1} .

Next, we shall prove the second equation in \eqref{eq:vansh_2_1}.
It is straightforward to show that $\nabla\mathcal{P}(\bff^{(1,2)}_j)({\bf 0})={\bf 0}$ by following the condition $\nabla \bff_j={\bf 0}$. We can then substitute this result into equation \eqref{eq:iden6} and multiply the resulting identity by $s^2$, yielding
{\small\begin{align*}
&\mu\, s^2\Re\mathcal{P}(\bff_+)({\bf 0})\cdot \begin{bmatrix}\mathrm{i}\\-1\end{bmatrix}\Big(e^{-s\sqrt{h}\,\hat{\mu}(\theta_M)}-e^{-s\sqrt{h}\,\hat{\mu}(\theta_m)}\Big)-2\begin{bmatrix}1\\\mathrm{i}\end{bmatrix}\cdot\bigg(\frac{\Re\mathcal{P}(\bg^{(1,2)}_+)({\bf 0})}{\hat{\mu}^2(\theta_M)}+\frac{\Re\mathcal{P}(\bg^{(1,2)}_-)({\bf 0})}{\hat{\mu}^2(\theta_m)}\bigg)\nonumber\\
&=\sum\limits_{j=1}^{10}s^2\,Q_j.
\end{align*}}
We note that the first term on the left hand of the last equation is bounded by $s^2e^{-c'_7s}$ with $c'_7=\sqrt{h}\min\{\cos\frac{\theta_M}{2},\cos\frac{\theta_m}{2}\}>0$. As $s$ tends to $+\infty$, we obtain
\beq\nonumber
\begin{bmatrix}1\\\mathrm{i}\end{bmatrix}\cdot\bigg(\frac{\Re\mathcal{P}(\bg^{(1,2)}_+)({\bf 0})}{\hat{\mu}^2(\theta_M)}+\frac{\Re\mathcal{P}(\bg^{(1,2)}_-)({\bf 0})}{\hat{\mu}^2(\theta_m)}\bigg)=0.
\eeq
By the same method of proving the  equation \eqref{eq:vansh_2_1}, we can prove the second equation in
\eqref{eq:vansh_2_1}.

\medskip
\noindent {\bf Part II.}~We next shall prove that
\begin{equation}\label{eq:vansh_3_1}
 \Re\mathcal{P}(f_+^{3})({\bf 0})=\Re\mathcal{P}(f_-^{3})({\bf 0})\quad\mbox{and}\quad \Re\mathcal{P}(g_+^{3})({\bf 0})=\Re\mathcal{P}(g_-^{3})({\bf 0})=0 ,
 \end{equation}
where the second equation need  more assumptions (i.e., $\nabla\mathcal{P}(\bff^{3}_+)({\bf 0})=\nabla\mathcal{P}(\bff^{3}_-)({\bf 0})={\bf 0}$).
In this part, we consider the PDE system~\eqref{eq:trans3D_2}. Let us deduce some similar operations above by using the CGO solution $u_0$ given in \eqref{cgo2}, we set up an integral identity as follows
 {\small \begin{align*}
 \frac{1}{\lambda}\int_{\mathcal{C}'_h}\big(\Re\mathcal{G}^{(3)}_1(\bx')-\Re\mathcal{G}^{(3)}_2(\bx')\big)\,&u_0\mathrm{d\bx'}=\int_{\Lambda'_h}\partial_{\bnu}\mathcal{P}(v_3-w_3)\,u_0-\partial_{\bnu}u_0\,\mathcal{P}(v_3-w_3)\mathrm{d\sigma}\\
&+\int_{\Gamma^{'+}_h}{\bnu_M}\cdot \big(\Re\mathcal{P}(f^{(1,2)}_+)(\bx') +\frac{1}{\mu}\Re\mathcal{P}(g^{3}_+)\big)\,u_0-\Re\mathcal{P}(f^{3}_+)\,\partial_{\bnu_M}u_0\mathrm{d\sigma}\\
&+\int_{\Gamma^{'-}_h}{\bnu_m}\cdot \big(\Re\mathcal{P}(f^{(1,2)}_-)(\bx') +\frac{1}{\mu}\Re\mathcal{P}(g^{3}_-)\big)\,u_0-\Re\mathcal{P}(f^{3}_-)\,\partial_{\bnu_m}u_0\mathrm{d\sigma}.
 \end{align*}}
Due to  the expansions of $\bff_\pm$ and $\bg_\pm$ in \eqref{expansion2}, the above integral identity can be reduced into

\begin{align*}
\lambda\,\mu\,\mathrm{i}\,\Big(\Re\mathcal{P}(f^3_+)({\bf 0}) -\Re\mathcal{P}(f^3_-)({\bf 0})-2\,\lambda\,s^{-1}\bigg(\frac{\Re\mathcal{P}(g^{3}_+)({\bf 0})}{\hat{\mu}^2(\theta_M)}+\frac{\Re\mathcal{P}(g^{3}_-)({\bf 0})}{\hat{\mu}^2(\theta_m)} \bigg)= \sum^{11}_{j=1}M_j,
\end{align*}
 where
\begin{align*}
&M_1=2\lambda \,s^{-1}\Re\mathcal{P}(g^{3}_+)({\bf 0})\Big(\hat{\mu}^{-1}(\theta_M)\,\sqrt{sh}\,e^{-\sqrt{sh}\,\hat{\mu}(\theta_M)}+\hat{\mu}^{-2}(\theta_M)\,e^{-\sqrt{sh}\,\hat{\mu}(\theta_M)}\Big),\\
&M_2=2\lambda \,s^{-1}\Re\mathcal{P}(g^{3}_-)({\bf 0})\Big(\hat{\mu}^{-1}(\theta_m)\,\sqrt{sh}\,e^{-\sqrt{sh}\,\hat{\mu}(\theta_m)}+\hat{\mu}^{-2}(\theta_m)\,e^{-\sqrt{sh}\,\hat{\mu}(\theta_m)}\Big),\\
&M_3=-\lambda\,\mu\,\int_{\Lambda'_h}\partial_{\bnu}u_0\,\,\Re\mathcal{P}(v_3-w_3)(\bx')  -\partial_{\bnu}u_0\, \,\Re\mathcal{P}(v_3-w_3)(\bx')\,\mathrm{d}\sigma,\\
&M_4=\lambda\,\mu\,\mathrm{i}\,\bigg(\frac{\Re\mathcal{P}(f^{3}_-)({\bf 0})}{e^{\sqrt{sh}\,\hat{\mu}(\theta_m)}}-\frac{\Re\mathcal{P}(f^{3}_+)({\bf 0})}{e^{\sqrt{sh}\,\hat{\mu}(\theta_M)}}\bigg),\qquad \qquad M_5=\mu\int_{\mathcal{C}'_h}(t_1+t_2)\,u_0\mathrm{d}{\bx'},\\
&M_6=\lambda\,\mu\,\int_{\Gamma'^+_h}\delta\Re\mathcal{P}(f^{3}_+)(\bx')\,\, \partial_{\bnu_M}u_0\,\mathrm{d}\sigma,\quad \quad M_7=\lambda\,\mu\,\int_{\Gamma'^-_h}\delta\Re\mathcal{P}(f^{3}_-)(\bx')\,\, \partial_{\bnu_m}u_0\,\mathrm{d}\sigma,\\
&M_{10}=\lambda\,\mu\,\int_{\Gamma'^+_h}\left(\nabla\Re\mathcal{P}(f^{3}_+)(\bx')\cdot \bx'\right)\,\, \partial_{\bnu_M}u_0\,\mathrm{d}\sigma,\,\,
M_8=-\lambda\int_{\Gamma'^+_h}\delta\Re\mathcal{P}(g^{3}_+)(\bx')\,\, u_0\,\mathrm{d}\sigma,\\
&M_{11}=\lambda\,\mu\,\int_{\Gamma'^-_h}\left(\nabla\Re\mathcal{P}(f^{3}_-)(\bx')\cdot \bx'\right)\,\, \partial_{\bnu_m}u_0\,\mathrm{d}\sigma,\,\,M_9=-\lambda\int_{\Gamma'^-_h}\delta\Re\mathcal{P}(g^{3}_-)(\bx')\,\, u_0\,\mathrm{d}\sigma.
\end{align*}
Here,
\begin{align*}
t_1&=-(2\lambda+\mu)\,\Re\int_{-L}^L\phi''(\bx')(v_3-w_3)\mathrm{d}x_3,\\
t_2&=(\lambda+\mu)\,\Re\int_{-L}^L\phi'(\bx')\big(\partial_1(v_1-w_1)+\partial_2(v_2-w_2)\big)\mathrm{d}x_3.
\end{align*}
Using those estimates list in Lemma~\ref{lem:cgo2}--Lemma~\ref{10-lem:23'}, we have
\begin{alignat*}{3}
&\big|M_1\big|=\Oh(e^{-q_1 \sqrt{s}}),&\quad&\big|M_2\big|=\Oh(e^{-q_2 \sqrt{s}}),&\quad&\big|M_3\big|=\Oh(s^{-1} e^{-1/2 h s}),\,\big|M_4\big|=\Oh(s^{-q_4\sqrt{s}}),\\
&\big|M_6\big|=\Oh(s^{-\alpha_+-1}),&\quad&\big|M_7\big|=\Oh(e^{-\alpha_--1}),&\quad&\big|M_8\big|=\Oh(e^{-\beta_+-1}),\,\,\quad\big|M_9\big|=\Oh(e^{-\beta_--1}),\\
&\big|M_{10}\big|=\Oh(s^{-1}),&\quad&\big|M_{11}\big|=\Oh(s^{-1}).&\quad&
\end{alignat*}
where these above constants do not depend on $s$. Using a similar technique of estimating $Q_8$, we get
$$
\big|M_5\big|=\Oh(e^{-q_5 s}),\quad \mbox{where} \quad q_5>0.
$$
Let $s \rightarrow +\infty$, it is direct to obtain the first equation in \eqref{eq:vansh_3_1}.

It is noted that $\nabla\mathcal{P}(\bff^{3}_j)({\bf 0})={\bf 0}$. We substitute the above equation into \eqref{eq:iden6} and then multiply the new identity by $s$, yielding
$$
\frac{\Re\mathcal{P}(g^{3}_+)({\bf 0})}{\mu^2(\theta_M)}+\frac{\Re\mathcal{P}(g^{3}_-)({\bf 0})}{\hat{\mu}^2(\theta_m)}=0.
$$
 It's worth noting that $\frac{\hat{\mu}^2(\theta_M)}{\hat{\mu}^2(\theta_m)}=e^{\mathrm{i}(\theta_M-\theta_m)}$ and $\theta_M-\theta_m\in (0,\pi)$. Therefore, the second equation in \eqref{eq:vansh_3_1} is valid.

Thanks to the symmetric roles of  $(\Re\bv,\Re\bw)$ and $(\Im\bv,\Im\bw)$, we can easily derive the two equations \eqref{eq:vansh2} and \eqref{eq:vansh2'}. 
 
The proof is complete.
\end{proof}

\section{The proofs of Theorem~\ref{th:main_loca}--Theorem~\ref{th:main2}}\label{se:proof}
\begin{proof}[Proof of Theorem~\ref{th:main_loca}]
We prove this theorem by contradiction. Assume that there exists a planar corner/3D edge corner $\bx_c$ on $\mathcal{S}_1\Delta\mathcal{S}_2$. Without loss of generality, we assume that $\bx_c$ coincides with the origin $\bf 0$, $\bf 0\in \mathcal{S}_1$ and $\bf 0\notin \mathcal{S}_2$. Denote $\bw=\bu_1-\bu_2$.

\medskip  \noindent $\bf Case~1$: $n=2$. We note that $\Gamma^\pm_{h}=\mathcal{S}_1\cap B_h$, $\mathcal{C}_{h}=\Omega_1\cap B_h$ and $\mathcal{C}_{h}\cap\Omega_2\neq \emptyset$ for sufficient small $h\in\mathbb{R}_+$, where $\Omega_1$ and $\Omega_2$ are defined in a similar way as \eqref{eq:omega1}. Since $\bu_1\big|_{\Sigma_0}=\bu_2\big|_{\Sigma_0}$ and $\Sigma_0\subset\Sigma_N$, we know that
$$
\bw=\bu_1-\bu_2=\mathcal{T}_{\bnu}\bu_1-\mathcal{T}_{\bnu}\bu_2={\bf 0}\quad \mbox{on}\quad \Sigma_0.
$$
Let $\bw^{-}$ and $\bw^{+}$ represent $\bw\big|_{\Omega_1}$ and $\bw\big|_{\Omega\backslash\overline{\Omega}_1}$, respectively. With the help of the unique continuation principle properly and the fact that $\bu_2$ is real analytic in $B_h$, it is clear to obtain
{\small\begin{equation*}
\mathcal{L}\,\bw^-={\bf 0}\,\,\mbox{in}\,\,\mathcal{C}_h,\,\,
\bw^-\big|_{\Gamma_h^j}=\bu_1\big|_{\Gamma_h^j}-\bu_2\big|_{\Gamma_h^j}=-\bff^1_j,\,\,
\mathcal{T}_{\bnu}\bu_1\big|_{\Gamma_h^j}-\mathcal{T}_{\bnu}\bu_2\big|_{\Gamma_h^j}=-\bg^1_j,\,\, j=+,-.
\end{equation*}}
From Proposition~\ref{prop:trans2}, we directly imply that $\bff^1_+({\bf 0})=\bff^1_-({\bf 0})$. 
This must contradict the admissibility condition $(5)$ in Definition ~\ref{def:Admis2}.

\medskip \noindent $\bf Case~2$: $n=3$. It is noted that ${\bf 0}=({\bf 0}', 0)^\top\in \Gamma'^\pm_{h}\times (-M,+M)\subset\mathcal{S}_1$ is a 3D edge corner point and $\mathcal{C}'_h\times (-M,+M)\subset \Omega_1$ and $\mathcal{C}'_h\times (-M,+M)\cap\Omega_2=\emptyset$, where $\Gamma'^\pm_{h}$ and $\mathcal{C}'_h$ are defined in \eqref{eq:ball1} and $\Omega_1,\,\Omega_2$ are given by \eqref{eq:omega1}. Similar to the previous case, we get
\beq\nonumber
\begin{cases}
\mathcal{L}\,\bw^-=\mathbf{0}, &\mbox{in}\quad\mathcal{C}'_h\times (-M,M),\\
\bw=-\bff^1_+&\mbox{on}\quad \Gamma'^+_h\times (-M,M),\\
\bw=-\bff^1_-&\mbox{on}\quad \Gamma'^-_h\times (-M,M),\\
\mathcal{T}_{\bnu_M}\bw=-\bg^1_+ &\mbox{on}\quad \Gamma'^+_h\times (-M,M),\\
\mathcal{T}_{\bnu_m}\bw=-\bg^1_- &\mbox{on}\quad \Gamma'^-_h\times (-M,M).
\end{cases}
\eeq
These results that $\bff^1_+({\bf 0})=\bff^1_-({\bf 0})$ 
is yielded by Proposition~\ref{prop:trans4}. This contradicts the admissibility condition $(5)$ in Definition~\ref{def:Admis2}.

The proof is complete.
\end{proof}

\begin{proof}[Proof of Theorem~\ref{th:main1}] Firstly, we shall prove $\mathcal{S}_1=\mathcal{S}_2$ by contradiction. Assume that $\Omega_1\neq\Omega_2$. Since $\Omega_1$ and $\Omega_2$ are both convex polygons or polyhedrons, there must exist a corner $\bx_c$ belonging to $\Omega_1\Delta\Omega_2$, which contradicts Theorem~\ref{th:main_loca}, thus we have $\Omega_1=\Omega_2$. It  directly leads to $\mathcal{S}_1=\mathcal{S}_2$. In what follows, we shall first prove \eqref{eq:unque1} and \eqref{eq:unque1'} for the 2D case.

Denote $\hat{\Omega}:=\Omega_1=\Omega_2$ and $\mathcal{S}:=\mathcal{S}_1=\mathcal{S}_2$. Let $\bw=\bu_1-\bu_2$. Since $\bu_1=\bu_2$ on $\Sigma_0\subset\Sigma_N$, as a direct consequence, we have $\bw=0$ and $\mathcal{T}_{\bnu}\bw=0$ on $\Sigma_0$, where $\bw=\bu_1-\bu_2$. By the unique continuation principle properly again, one obtains
$$
\bw^+\big|_{\Gamma_h^\pm}=\mathcal{T}_{\bnu}\bw^+\big|_{\Gamma_h^\pm}={\bf 0},
$$
where $\bw^{+}$ represents $\bw\big|_{\Omega\backslash\overline{\hat{\Omega}}}$.
Since $(\mathcal{S};\bff^1,\bg^1)$ and $(\mathcal{S};\bff^2,\bg^2)$ are admissible, we get
\beq\label{eq:c1}
\mathcal{L}\,\bw^-={\bf 0}\,\,\mbox{in}\,\,\mathcal{C}_h,\quad
\bw^-\big|_{\Gamma_h^j}=\bff^2_j-\bff^1_j\,\,\,\mbox{and}\,\,\,
\mathcal{T}_{\bnu}\bw^-\big|_{\Gamma_h^j}=\bg^2_j-\bg^1_j\quad\mbox{for}\quad j=+,-,
\eeq
where $\bw^{-}$ signifies $\bw\big|_{\hat{\Omega}}$.
From Proposition~\ref{prop:trans2} and \eqref{eq:c1}, we obtain the following local uniqueness 
$$
\bff^2_+({\bf 0})-\bff^1_+({\bf 0})=\bff^2_-({\bf 0})-\bff^1_-({\bf 0}).
$$
Due to $\bff^i$ ($i=1,2$) being a piecewise constant valued function, we obtain the first equation in \eqref{eq:unque1}.
Combing  this with the conditions  $\nabla \bff^1({\bf 0})=\nabla \bff^2({\bf 0})$ holds at the corner point ${\bf 0}$, the fact that $\bg^i$($i=1,2$) is a piecewise constant valued  function, and Proposition~\ref{prop:trans2}, we immediately prove that \eqref{eq:unque1'} holds.

Indeed, we can
use a similar method in the 2D case to prove \eqref{eq:unque1} and \eqref{eq:unque1'} in the 3D case. Therefore, we omit the details.
\end{proof}
\begin{proof}[Proof of Theorem~\ref{th:main2}]
The argument for proving this theorem is similar to the one used in the proof of Theorem~\ref{th:main1}, where we only need some necessary modifications. Let $\mathcal{S}_1$ and $\mathcal{S}_2$ are two different linear piecewise curves in $\mathbb R^2$ or polyhedral surfaces in $\mathbb R^3$. From Definition \ref{def:poly}, we can directly imply that $\mathcal{S} _1 \Delta \mathcal{S} _2$ contains a planar or 3D edge corner. Under the condition \eqref{eq:thm33}, adopting a similar argument in proving Theorem \ref{th:main1}, we can show that $\mathcal{S} _1= \mathcal{S} _2$.

Set $\bw=\bu_1-\bu_2$. We have that $\bw=\mathcal{T}_{\bnu}\bw=\bf 0$ on $\Sigma_0$ for $\bu_1=\bu_2$ on $\Sigma_0\subset\Sigma_N$. We again use the unique continuation property to conclude that $\bw=\bf 0$ in $\Omega\backslash \overline{\mathcal{S}}$. Hence, it is direct to imply that
$$
\bf 0=[\bw]_{\mathcal{S}_1}=[\bw]_{\mathcal{S}_2}\,\Rightarrow\, [\bu_1]_{\mathcal{S}_1}=[\bu_2]_{\mathcal{S}_2}\,\mbox{and}\,[\mathcal{T}_{\bnu}\bu_1]_{\mathcal{S}_1}=[\mathcal{T}_{\bnu}\bu_2]_{\mathcal{S}_2}\,\Rightarrow\, \bff^1=\bff^2 \,\mbox{and}\,\bg^1=\bg^2.
$$

The proof is complete.
\end{proof}

\section*{Acknowledgment}
The work of H. Diao is supported by the National Natural Science Foundation of China  (No. 12371422) and the Fundamental Research Funds for the Central Universities, JLU (No. 93Z172023Z01).  The work of H. Liu was supported by the Hong Kong RGC General Research Funds (projects 11311122, 11300821, and 11303125), NSF/RGC Joint Research Fund (project N\_CityU101/21), and the ANR/RGC Joint Research Fund (project A\_CityU203/19). The work of Q. Meng is supported by the Hong Kong RGC Postdoctoral Fellowship (No.: PDFS2324-1S09).

\end{document}